\documentclass[10pt]{amsart}
\usepackage{amsmath,amssymb,amsthm}
\usepackage[capitalise]{cleveref}
\usepackage{mathtools}
\usepackage{url}

\usepackage[letterpaper,hmargin=1in,vmargin=1in]{geometry}
\newtheorem{proposition}{Proposition}
\newtheorem*{remark}{Remark}
\newtheorem*{definition}{Definition}
\newtheorem{lemma}{Lemma}

\newtheorem{theorem}{Theorem}

\DeclareMathOperator{\cone}{cone}

\newcommand{\Tr}{\mathrm{Tr}\ }

\newcommand{\eps}{\varepsilon}
\newcommand{\HH}{\mathbb H}
\newcommand{\BH}{\mathsf B}
\newcommand{\PH}{\mathsf P}
\newcommand{\calC}{\mathcal C}


\begin{document}

\title{The heat trace for domains with curved corners}

\author[S. Looi]{Sam Looi}\address{California Institute of Technology, The Division of Physics, Mathematics and Astronomy,
1200 E California Blvd, Pasadena CA 91125, USA}
\email{looi@caltech.edu}
\author[D. Sher]{David Sher}
\address{DePaul University, Department of Mathematical Sciences, 2320 N. Kenmore Ave., Chicago IL 60614, USA}
\email{dsher@depaul.edu}

\begin{abstract}
The heat trace of a planar polygon contains corner terms depending only on the opening angles, while the heat trace of a smooth planar domain contains curvature terms along the boundary. We show that, for curvilinear polygons, these two phenomena first interact at order $t^{1/2}$. We compute this first corner-curvature heat invariant and prove a sharp sign law for its Dirichlet angular factor: its sign is determined solely by whether the corner is convex or reflex. More precisely, we derive the local heat trace expansion through order $t^{1/2}$, for both Dirichlet and Neumann boundary conditions. The new coefficient decomposes into the usual smooth-boundary contribution and a sum of local curved-corner terms, each depending only on the interior angle $\alpha$ and the one-sided limiting curvatures $\kappa_{\pm}$ of the adjacent arcs. In the Dirichlet case, the curved-corner contribution has the form $\mathcal C_{1/2}(\alpha,\kappa_+,\kappa_-)
=
c_{1/2}(\alpha)\frac{\kappa_+ + \kappa_-}{4\sin(\alpha/2)}$, with $c_{1/2}(\alpha)$ given by an explicit sector heat kernel integral. We determine its sign for every $0<\alpha<2\pi$ and compute it explicitly when $\alpha=\pi/N$ for a natural number $N$. The sign law has a spectral consequence: it gives a new obstruction to a curvilinear polygon being Dirichlet isospectral to a straight-sided polygon. In particular, every convex curvilinear polygon which is Dirichlet isospectral to a straight-sided polygon must itself be straight-sided, removing the assumption of straight corners from the theorem of Enciso and Gómez-Serrano.
\end{abstract}

\maketitle

\section{Introduction}

A classical problem in spectral geometry is to determine what geometric information is encoded in the spectrum of the Laplacian \cite{kac1966can}. One of the main tools is the short-time expansion of the heat kernel trace. Since the trace is spectrally determined, every coefficient in this expansion is a spectral invariant. For smooth planar domains, McKean and Singer showed that these invariants include area, perimeter, and the Euler characteristic \cite{mckean1967curvature}. 

Geometric singularities in the underlying manifold or its boundary add local terms to the heat trace. For domains with corners, these terms are concentrated at the vertices. In the exact polygonal setting, the vertex terms were first computed in unpublished work of Ray, as cited in \cite{Cheeger1983,vanDenBergSrisat88}; the expansion was extended to curvilinear polygons in \cite{NRS}. Curvilinear polygons lie between the exact polygonal and smooth boundary cases: their boundaries have both curvature along the smooth arcs and angle singularities at the vertices. 
For exact polygonal domains, the vertex contribution through the constant term is a function only of the opening angle, and the same angle-only behavior persists for curvilinear polygonal corners at order $t^0$. Curvature along the regular arcs enters through the smooth-boundary terms. The coefficient of $t^{1/2}$ is therefore the first heat coefficient at which boundary curvature and corner geometry can interact. This paper computes the corresponding local curved-corner term. 

Let $\Omega \subset \mathbb{R}^2$ be a curvilinear polygon whose boundary is smooth except at finitely many corners $P_j$, with interior angles $\alpha_j > 0$, $j=1,\ldots,n$. Let $\kappa$ denote the inward-pointing curvature of the boundary, and let $\gamma(s)$ be an arc-length parametrization of $\partial \Omega$, oriented counterclockwise. For each $j$, let $\kappa_{j,+}$ and $\kappa_{j,-}$ be the one-sided limits of $\kappa(\gamma(s))$ as $s$ approaches $\gamma^{-1}(P_j)$ from above and below.  We show that at order $t^{1/2}$, the heat trace contains the familiar smooth-boundary contribution, a multiple of
\[
\int_{\partial \Omega} \kappa^2 ds,
\]
and a sum of local curved-corner contributions, 
\[ \mathcal C_{1/2}(\alpha_j,\kappa_{j,+},\kappa_{j,-}). \]
The first main theorem separates these two parts and shows that each corner contribution depends only on the opening angle and the two one-sided limiting curvatures of the adjacent arcs. 

\begin{theorem}\label{thm:thm1}
With notation as above, the Dirichlet heat trace $H^D_\Omega(t)$ has an expansion as $t \to 0$ given by
\begin{equation}\label{eq:heattrace}
\begin{aligned}
H^D_\Omega(t)
&=
\frac{|\Omega|}{4\pi t} - \frac{|\partial \Omega|}{8\sqrt{\pi t}} + \frac{1}{12\pi} \left( \int_{\partial \Omega} \kappa ds + \sum_{j=1}^n \frac{\pi^2-\alpha_j^2}{2\alpha_j} \right)
\\
&\quad
+ \sqrt{t} \left( \frac{1}{256\sqrt{\pi}} \int_{\partial \Omega} \kappa^2 ds + \sum_{j=1}^n \mathcal C_{1/2}(\alpha_j,\kappa_{j,+},\kappa_{j,-}) \right) + O(t\log t),
\end{aligned}
\end{equation}
where $\mathcal C_{1/2}$ is a function only of the angle $\alpha$ and the limiting curvatures $\kappa_\pm$.
\end{theorem}

We emphasize that the coefficient $\mathcal C_{1/2}$ computed here is a genuinely mixed corner curvature invariant. It is invisible in the exact-sector model, where $\kappa_{j,+}=\kappa_{j,-}=0$, and it is also invisible in the smooth-boundary expansion, whose local coefficients have no vertex contribution. In the ``straight corners'' case, where a neighborhood of each corner $P_j$ is isometric to an exact sector, this expansion has been computed in \cite{EGS_2017} using results from \cite{branson1990asymptotics}: there is no sum of $\mathcal C_{1/2}$ terms there. 

Our second main theorem gives a general formula for the Dirichlet corner contribution, valid for arbitrary opening angle and arbitrary one-sided limiting curvatures.  In particular, only the bisector-even combination $\kappa_+ + \kappa_-$ occurs.

For $0<\alpha<2\pi$, write \[ W_\alpha=\{re^{i\theta}:r>0,\ 0<\theta<\alpha\}, \qquad \HH=W_\pi, \] and let $q_s$ denote the point at distance $s$ from the vertex on the boundary ray $\theta=0$.  For a domain $\Omega$ with Dirichlet heat kernel $H_\Omega(t;z,z')$ and a regular boundary point $q\in\partial\Omega$, set \[ \mathsf B_\Omega(t,q) := \int_\Omega \left( \partial_{\nu_q}H_\Omega(t/2;q,z) \right)^2\,dz . \]
Let 
$$J_\alpha:=\int_0^\infty s^2
\left(
    \mathsf B_{\HH}(1,q_s)
    -
    \mathsf B_{W_\alpha}(1,q_s)
\right)\,ds .$$
\begin{theorem}\label{thm:moreid}
For $\alpha\in(0,2\pi)$, define the angle coefficient
\begin{equation}\label{eq:moreexplicitc}
c_{1/2}(\alpha)
:=
2\sin(\alpha/2)
\int_0^\infty s^2
\left(
    \mathsf B_{\HH}(1,q_s)
    -
    \mathsf B_{W_\alpha}(1,q_s)
\right)\,ds .
\end{equation}
Then, for all $\kappa_+,\kappa_-\in\mathbb R$,
\begin{equation}\label{eq:form}
\mathcal C_{1/2}(\alpha,\kappa_+,\kappa_-)
=
\frac{\kappa_+ + \kappa_-}{4\sin(\alpha/2)}
\,c_{1/2}(\alpha).
\end{equation}
Equivalently,
\begin{equation}\label{eq:C12-direct}
\mathcal C_{1/2}(\alpha,\kappa_+,\kappa_-)
=
\frac{\kappa_+ + \kappa_-}{2}
\int_0^\infty s^2
\left(
    \mathsf B_{\HH}(1,q_s)
    -
    \mathsf B_{W_\alpha}(1,q_s)
\right)\,ds .
\end{equation}

Moreover, with $\operatorname{sgn}(0):=0$,
\begin{equation}\label{eq:sign}
\operatorname{sgn} c_{1/2}(\alpha)
=
\operatorname{sgn}(\pi-\alpha),
\qquad \alpha\in(0,2\pi).
\end{equation}
If $N\ge 1$ and $\alpha=\pi/N$, then 
\begin{equation}\label{eq:I-pi-over-N}
J_{\pi/N}=\frac{1}{16\sqrt{\pi}}
\sum_{k=1}^{N-1}
\frac{1+\cos^2(\pi k/N)}
{\sin^3(\pi k/N)} ,
\end{equation}
where the sum is interpreted as $0$ when $N=1$. Hence
\begin{equation}\label{eq:c-pi-over-N}
c_{1/2}\left(\frac{\pi}{N}\right)
=
\frac{\sin(\pi/(2N))}{8\sqrt{\pi}}
\sum_{k=1}^{N-1}
\frac{1+\cos^2(\pi k/N)}
{\sin^3(\pi k/N)} .
\end{equation}
In particular,
\begin{equation}\label{eq:specialpi2}
c_{1/2}\left(\frac{\pi}{2}\right)
=
\frac{\sqrt{2}}{16\sqrt{\pi}},
\end{equation}
and hence
\begin{equation}\label{eq:C12-pi2}
\mathcal C_{1/2}\left(\frac{\pi}{2},\kappa_+,\kappa_-\right)
=
\frac{\kappa_+ + \kappa_-}{32\sqrt{\pi}} .
\end{equation}
\end{theorem}

\begin{remark} A few comments:
\begin{itemize}
\item The $O(t\log t)$ error in \eqref{eq:heattrace} is not necessarily optimal, and in fact we expect that it should be $O(t)$.
\item For the inverse spectral results at a non-straight corner, given below, the relevant fact is the nonvanishing of $c_{1/2}(\alpha)$. More precisely, \eqref{eq:sign} implies that $c_{1/2}(\alpha)\neq 0$ for every $0<\alpha<2\pi$ with $\alpha\neq\pi$, while $c_{1/2}(\pi)=0$. 
\item An alternative formula for $\BH_{\Omega}(t,q)$, with two derivatives but no integrals, is given by \eqref{eq:B-normal}.  By \eqref{eq:B-normal}, the quantity $J_\alpha$ can also be written without an integral over the domain:
\begin{equation}\label{eq:I-alpha-normal}
\begin{aligned}
J_\alpha
&=
\int_0^\infty s^2
\left[
\partial_{\nu_x}\partial_{\nu_y}H_{\HH}(1;x,y)\big|_{x=y=q_s}
-
\partial_{\nu_x}\partial_{\nu_y}H_{W_\alpha}(1;x,y)\big|_{x=y=q_s}
\right]ds  \\
&=
\int_0^\infty s^2
\left[
\frac{1}{4\pi}
-
\partial_{\nu_x}\partial_{\nu_y}H_{W_\alpha}(1;x,y)\big|_{x=y=q_s}
\right]ds .
\end{aligned}
\end{equation}
Here the normal derivatives are taken with respect to the outward normal along the ray $\theta=0$. The last expression is read as the displayed difference; the two terms do not define convergent integrals separately.
\item The conformal model that we use for $\alpha\ne\pi$ breaks down in the straight-angle case, as can be seen from the fact that the parameter $r_0$ in \eqref{eq:defofr0} is undefined when $\alpha=\pi$. However, \eqref{eq:sign} shows that a straight-angle vertex has no additional local contribution at order $t^{1/2}$.

\item In \eqref{eq:heattrace}, the corner contributions to the coefficients of $t^{-1}$ and $t^{0}$
depend only on the interior angles $\alpha_j$; the one-sided curvatures $\kappa_{j,\pm}$ do not appear.
Theorem~\ref{thm:moreid} shows that the first place where the corner contribution becomes sensitive to
$(\kappa_{j,+},\kappa_{j,-})$ is at order $t^{1/2}$.
\end{itemize}
\end{remark}

\subsection{The Neumann case}

The same local analysis gives the corresponding expansion for the Neumann heat trace. The first four terms in the Neumann heat trace expansion in the straight-corners setting are given in \cite{EGS_2017}.

\begin{proposition}
Under the same assumptions as in the Dirichlet case, the Neumann heat trace satisfies
\begin{equation}
\begin{aligned}
H^N_\Omega(t) &= \frac{|\Omega|}{4\pi t} + \frac{|\partial \Omega|}{8\sqrt{\pi t}} + \frac{1}{12\pi} \left( \int_{\partial \Omega} \kappa ds + \sum_{j=1}^n \frac{\pi^2-\alpha_j^2}{2\alpha_j} \right)
\\
&\quad + \sqrt{t} \left( \frac{5}{256\sqrt{\pi}} \int_{\partial \Omega} \kappa^2 ds + \sum_{j=1}^n \mathcal C^N_{1/2}(\alpha_j,\kappa_{j,+},\kappa_{j,-}) \right) + O(t\log t).
\end{aligned}
\end{equation}
\end{proposition}

We expect an analogue of the Dirichlet factorization for $\mathcal C^N_{1/2}$, but the argument in the Neumann case has additional boundary terms and we do not pursue it here.

\subsection{Inverse spectral applications}

The coefficient above has a direct inverse spectral consequence. Enciso and G\'omez-Serrano proved that any convex curvilinear polygon with straight corners that is Dirichlet-isospectral to a polygon must itself be a polygon \cite{EGS_2017}. The straight-corner hypothesis was needed because the curved-corner coefficient at order $t^{1/2}$ was not available. The formula above removes that missing input.

\begin{definition}
A curvilinear polygon $\Omega$ is admissible if, for each corner $P_j$ with opening angle $\alpha_j \neq \pi$,
\[
c_{1/2}(\alpha_j)(\kappa_{j,+}+\kappa_{j,-}) \geq 0.
\]
\end{definition}

Since $\sin(\alpha_j/2)>0$ for $0<\alpha_j<2\pi$, admissibility is equivalent to requiring that each non-straight corner term in \eqref{eq:form} be nonnegative. Straight angle points contribute zero by \eqref{eq:sign}.

The admissibility condition admits a clean geometric reading in light of \eqref{eq:sign}. Since $c_{1/2}(\alpha_j)>0$ for $\alpha_j\in(0,\pi)$ and $c_{1/2}(\alpha_j)<0$ for $\alpha_j\in(\pi,2\pi)$, admissibility reduces to the corner-by-corner sign condition
\[
\kappa_{j,+}+\kappa_{j,-}\ge 0 \quad \text{at convex corners } (\alpha_j<\pi),
\]
\[
\kappa_{j,+}+\kappa_{j,-}\le 0 \quad \text{at reflex corners } (\alpha_j>\pi).
\]
In particular, \textbf{every convex curvilinear polygon is admissible}: convexity gives $\alpha_j\in(0,\pi)$ and $\kappa_{j,+},\kappa_{j,-} \geq 0$ at every corner, so the condition holds automatically. Admissibility is therefore a strictly weaker hypothesis than convexity, while a gap between them is meaningful at reflex corners.

\begin{proposition}\label{prop:inverse_spectral_application}
    Suppose that an admissible curvilinear polygon $\Omega$ is Dirichlet isospectral to a polygon. Then $\Omega$ is a polygon.
\end{proposition}
\begin{proof}
For a polygon, the $t^{1/2}$ term in the heat trace vanishes. So it must vanish for $\Omega$. But under our assumptions, by Theorems \ref{thm:thm1} and \ref{thm:moreid}, the $t^{1/2}$ coefficient in $H^D_{\Omega}(t)$ is given by 
    \begin{equation}\label{eq:Dirichlet_iso}
    \frac{1}{256\sqrt\pi}\int_{\partial\Omega}\kappa^2\, ds + \sum_{j=1}^n c_{1/2}(\alpha_j)\,\frac{\kappa_{j,+}+\kappa_{j,-}}{4\sin(\alpha_j/2)}.
    \end{equation}
Note that $\int_{\partial\Omega}\kappa^{2}\,ds\ge 0$. By admissibility, every corner term in \eqref{eq:Dirichlet_iso} is also nonnegative. Since the full $t^{1/2}$ coefficient vanishes, each term in \eqref{eq:Dirichlet_iso} must vanish. In particular,
\[
\int_{\partial\Omega}\kappa^2\,ds=0.
\]
Hence $\kappa\equiv 0$ on each smooth edge, so the one-sided limits $\kappa_{j,+}$ and $\kappa_{j,-}$ also vanish at every corner. Therefore $\partial\Omega$ is piecewise linear and $\Omega$ is a polygon.
\end{proof}

We record the generalization of the result of Enciso and G\'omez-Serrano \cite{EGS_2017} as an explicit corollary:
\begin{proposition}
    Suppose that $\Omega$ is a convex curvilinear polygon which is Dirichlet isospectral to a polygon. Then $\Omega$ is a polygon.
\end{proposition}
In particular we may remove their straight-corners hypothesis.

\subsection{Circular and annular sectors}

We state two model families in which the coefficient of $t^{1/2}$ in \eqref{eq:heattrace} admits a closed expression. These formulas give concrete examples of the contribution of curved corners to the heat trace at order $t^{1/2}$.

For a curvilinear polygon $\Omega$, write
\begin{equation}\label{eq:def-a12}
H^D_{\Omega}(t)= \cdots + a_{1/2}(\Omega)\,\sqrt t + O(t\log t),
\qquad
a_{1/2}(\Omega)=\frac{1}{256\sqrt{\pi}}\int_{\partial\Omega}\kappa^2\,ds
+\sum_{j=1}^n\mathcal{C}_{1/2}(\alpha_j,\kappa_{j,+},\kappa_{j,-}).
\end{equation}
In the examples below, every curved corner is a right angle. We therefore use the special value
\begin{equation}\label{eq:c12-pi2}
c_{1/2} \left(\frac{\pi}{2}\right)=\frac{\sqrt{2}}{16\sqrt{\pi}}.
\end{equation}
Since $\sin(\pi/4)=\sqrt{2}/2$, \eqref{eq:form} reduces in this case to
\begin{equation}
\mathcal{C}_{1/2} \left(\frac{\pi}{2},\kappa_{+},\kappa_{-}\right)
=\frac{\kappa_+ + \kappa_-}{32\sqrt{\pi}}.
\end{equation}

\textbf{Circular sectors.}
Let $R>0$ and $\alpha\in(0,2\pi)$, and set
\[
S(R,\alpha)=\{(r,\theta)\in\mathbb{R}^2:\ 0<r<R,\ 0<\theta<\alpha\}.
\]
The boundary of $S(R,\alpha)$ consists of two radial segments and one circular arc. There are three
corners: the apex with opening angle $\alpha$, and two right-angle corners at the endpoints of the arc.

The basic geometric quantities are
\begin{equation}\label{eq:sector-geom}
|S(R,\alpha)|=\frac{\alpha}{2}R^2,
\qquad
|\partial S(R,\alpha)|=(\alpha+2)R,
\qquad \int_{\partial S(R,\alpha)}\kappa\, ds = \alpha,\qquad
\int_{\partial S(R,\alpha)}\kappa^2\,ds=\frac{\alpha}{R}.
\end{equation}
The formula \eqref{eq:heattrace} then becomes, after a bit of simplification,
\begin{equation}
    H^D_{S(R,\alpha)}(t) = \frac{\alpha R^2}{8\pi t} - \frac{(\alpha +2)R}{8\sqrt{\pi t}} + \frac{1}{12\pi}\Big(\frac{3\pi}{2} + \frac{\pi^2+\alpha^2}{2\alpha}\Big)+a_{1/2}(S(R,\alpha))\sqrt t+O(t\log t).
\end{equation}

To find $a_{1/2}$, observe that at the apex, $(\kappa_{+},\kappa_{-})=(0,0)$, hence $\mathcal{C}_{1/2}=0$. At each endpoint of the
circular arc, the corner angle is $\pi/2$ and the limiting curvatures are $(\kappa_{+},\kappa_{-})=(1/R,0)$,
so \eqref{eq:C12-pi2} gives a contribution $1/(32\sqrt{\pi}R)$. Summing the two endpoints yields
\begin{equation}\label{eq:sector-corner-sum}
\sum_{j}\mathcal{C}_{1/2}(\alpha_j,\kappa_{j,+},\kappa_{j,-})
=\frac{1}{16\sqrt{\pi}}\cdot \frac{1}{R}.
\end{equation}
Substituting \eqref{eq:sector-geom} and \eqref{eq:sector-corner-sum} into \eqref{eq:def-a12} gives
\begin{equation}\label{eq:sector-a12}
a_{1/2} \bigl(S(R,\alpha)\bigr)
=\frac{1}{256\sqrt{\pi}}\left(\frac{\alpha}{R}\right)
+\frac{1}{16\sqrt{\pi}}\left(\frac{1}{R}\right)
=\frac{\alpha+16}{256\sqrt{\pi}}\cdot \frac{1}{R},
\end{equation}
which completes our calculation.

Suppose that two circular sectors are Dirichlet isospectral. Then they must be congruent; their congruence can already be deduced from the coefficients of orders $t^{-1}$, $t^{-1/2}$ and $t^0$. We give an alternative proof here to show that the coefficients $a_{-1}$ and $a_{1/2}$ are enough by themselves.

\begin{proposition}[Spectral determination within the sector family]\label{prop:sector-rigid}
If two circular sectors $S(R,\alpha)$ and $S(\widetilde R,\widetilde\alpha)$ are Dirichlet isospectral,
then $(R,\alpha)=(\widetilde R,\widetilde\alpha)$.
\end{proposition}

\begin{proof}
Dirichlet isospectrality implies equality of the heat invariants $|\partial\Omega|$ and $a_{1/2}(\Omega)$.
Using \eqref{eq:sector-geom} and \eqref{eq:sector-a12}, we compute
\[
|\partial S(R,\alpha)|\cdot 256\sqrt{\pi}\,a_{1/2} \bigl(S(R,\alpha)\bigr)
=(\alpha+2)R\cdot \frac{\alpha+16}{R}
=(\alpha+2)(\alpha+16).
\]
Hence $\alpha$ is determined as the unique positive solution of
\[
x^2+18x+32=|\partial\Omega|\cdot 256\sqrt{\pi}\,a_{1/2}(\Omega),
\]
and then $R=|\partial\Omega|/(\alpha+2)$. Uniqueness of $\alpha$ follows since $x^2+18x+32$ is strictly
increasing on $(0,\infty)$.
\end{proof}

\textbf{Annular sectors.}
Let
\[
A(r,R,\alpha)=\{(\rho,\theta)\in\mathbb R^2:\ r<\rho<R,\ 0<\theta<\alpha\}.
\]
Then
\begin{equation}\label{eq:annular-aearly}
|A(r,R,\alpha)|=\frac{\alpha}{2}(R^2-r^2),
\qquad
|\partial A(r,R,\alpha)|=\alpha(R+r)+2(R-r).
\end{equation}
For annular sectors the coefficient $a_0$ carries no additional information: one has $\int_{\partial\Omega}\kappa\,ds=0$ and all four corner angles equal $\pi/2$, hence $a_0=\tfrac14$ for every $A(r,R,\alpha)$. Hence among the heat invariants through order $\sqrt t$, the only geometrically nontrivial data are $(|\Omega|,|\partial\Omega|,a_{1/2}(\Omega))$.

Thus we have 
\begin{equation}
    H^D_{A(r,R,\alpha)}(t)=\frac{\alpha(R^2-r^2)}{8\pi t} - \frac{\alpha(R+r)+2(R-r)}{8\sqrt{\pi t}} + \frac 14 + a_{1/2}(A(r,R,\alpha))\sqrt t+O(t\log t).
\end{equation}
Using \eqref{eq:annular-aearly}, together with the fact that the two outer corners have
$(\kappa_+,\kappa_-)=(1/R,0)$ and the two inner corners have $(\kappa_+,\kappa_-) = (-1/r,0)$,
we obtain from \eqref{eq:C12-pi2}
\begin{equation}\label{eq:annular-a12}
a_{1/2}\bigl(A(r,R,\alpha)\bigr)
=\frac{1}{256\sqrt{\pi}}\left(\frac{\alpha+16}{R}+\frac{\alpha-16}{r}\right).
\end{equation}

This additional coefficient may be used to distinguish annuli which have equal heat coefficients up to that point. For example, set
\[
\Omega_1=A \left(1,3,\frac14\right),
\qquad
\Omega_2=A \left(2,4,\frac16\right).
\]
A direct check gives
\[
|\Omega_1|=|\Omega_2|=1,
\qquad
|\partial\Omega_1|=|\partial\Omega_2|=5,
\]
so $\Omega_1$ and $\Omega_2$ have the same heat coefficients through $t^0$ within this model family.
However,
\[
a_{1/2}(\Omega_1)
=\frac{1}{256\sqrt{\pi}}\left(\frac{16+\frac14}{3}+\left(\frac14-16\right)\right)
=-\frac{31}{768\sqrt{\pi}},
\qquad
a_{1/2}(\Omega_2)
=\frac{1}{256\sqrt{\pi}}\left(\frac{16+\frac16}{4}+\frac{\frac16-16}{2}\right)
=-\frac{31}{2048\sqrt{\pi}},
\]
so $a_{1/2}(\Omega_1)\neq a_{1/2}(\Omega_2)$. Hence $a_{1/2}$ distinguishes noncongruent annular sectors
that are indistinguishable by the coefficients through order $t^0$.

That said, even after including $a_{1/2}$, the data $(A=|\Omega|,P=|\partial\Omega|,a_{1/2})$ can arise from two distinct triples $(r,R,\alpha)$,
hence from two noncongruent annular sectors.
To see this, set $s=R+r$ and $d=R-r$. We have $A=\frac{\alpha}{2}ds$ and $P=\alpha s+2d$, so eliminating $\alpha$
gives
\begin{equation}\label{eq:annular-elim-alpha}
P=\frac{2A}{d}+2d.
\end{equation}
For fixed $(A,P)$ with $P>4\sqrt{A}$, the quadratic $2d^2-Pd+2A=0$ has two positive roots, so $d$ is not
uniquely determined by $(A,P)$.
Moreover, from \eqref{eq:annular-a12} we deduce
\begin{equation}\label{eq:annular-s-equation}
\bigl(256\sqrt{\pi}\,a_{1/2}(\Omega)\bigr)\,(s^2-d^2)=\frac{8A}{d}-64d.
\end{equation}
For each fixed $d>0$, this determines $s^2$, and hence determines the unique value $s>d$.
Thus each admissible choice of $d$ determines a unique triple $(r,R,\alpha)$.

Therefore, even after incorporating the curved-corner invariant in Theorem~\ref{thm:moreid}, the data
$(A,P,a_{1/2})$ can correspond to two noncongruent annular sectors. 
We expect---but do not prove here---that the two annular sectors with equal $(A, P, a_{1/2})$ are always distinguished by the first Dirichlet eigenvalue. It would be interesting to determine whether the family $\{A(r, R, \alpha)\}$ admits genuine Dirichlet-isospectral pairs.

\subsection{Plan of the paper}

In order to understand the interaction of curvature and corners, we need a precise description of the short-time heat kernel near such a corner. The heat kernel has different
short-time behaviors in different geometric regimes: near an interior point,
near a smooth boundary point, and near a corner at the scale
$\operatorname{dist}(z,P)\sim \sqrt t$. In section \ref{sec:gm}, we use geometric microlocal analysis to
separate these regimes cleanly. Concretely, one replaces the usual heat-kernel
space by a resolved heat space in which the collision of the two spatial
variables with the corner, as $t\to0$, is replaced by a new boundary component
recording the rescaled variables
\[
R=\frac{r}{\sqrt t},\qquad R'=\frac{r'}{\sqrt t}.
\]
Thus the corner is not treated as an unresolved singular point; it is replaced
by a model problem on the tangent sector. In the standard terminology this new
boundary component is the corner front face, but the important point is simply
that it isolates the heat flow occurring at the parabolic scale of the corner.

This separation has two uses. First, it gives a locality theorem for the
order-$t^{1/2}$ corner term. If two curvilinear corners have the same opening
angle and the same one-sided limiting curvatures, then their heat kernels have
the same expansion through the orders that can contribute to the local
$t^{1/2}$ corner coefficient. The proof constructs an approximate heat kernel
whose model terms agree in the interior, along the smooth boundary arcs, and in
the rescaled corner variables. The equality of
$(\alpha,\kappa_+,\kappa_-)$ is precisely what makes these model terms
compatible where the smooth-boundary and corner regimes meet. The remaining
error is then removed by a Volterra series, and the heat-calculus composition
theorem shows that the correction terms are too high order to affect the
coefficient under consideration. Hence the new corner coefficient depends only
on $(\alpha,\kappa_+,\kappa_-)$. This is carried out in Section \ref{sec:gm}.

Second, locality allows the coefficient to be computed in a convenient model, which we do in section \ref{sec:conformal}.
For $\alpha\ne\pi$ we replace a general curved corner by the image of an exact
sector under the quadratic conformal map
\[
w=z+z_0z^2.
\]
The parameter $z_0$ realizes arbitrary one-sided limiting curvatures
$\kappa_+$ and $\kappa_-$. Pulling back to the exact sector changes the
operator to
\[
F^{-1}\Delta,\qquad F=|1+2z_0z|^2.
\]
Thus the angular singularity is still represented by the exact sector heat
kernel, while the first curvature jet appears as the first perturbation of the
operator. We construct a sector parametrix incorporating this perturbation and
extract the finite-part trace coefficient. An argument using reflection symmetry across the
bisector eliminates the odd curvature combination and gives the factorization
\[
\mathcal C_{1/2}(\alpha,\kappa_+,\kappa_-)
=
c_{1/2}(\alpha)\frac{\kappa_+ + \kappa_-}{4\sin(\alpha/2)},
\]
along with an explicit but unwieldy formula for $c_{1/2}(\alpha)$ involving a Hadamard finite-part integral.

In section \ref{sec:specialpi2}, we prove \eqref{eq:specialpi2}, i.e. compute $c_{1/2}(\pi/2)$, by computing a sufficient number of terms in the Dirichlet heat trace expansion for a unit half-disk and comparing with \eqref{eq:heattrace}. Section \ref{sec:specialpi2} also contains the corresponding calculation for the Neumann case, identifying $\mathcal C^N_{1/2}\bigl(\frac{\pi}{2},1,0\bigr)$.

Finally, in section \ref{section:sign}, we compute the general formula \eqref{eq:moreexplicitc} and prove \eqref{eq:sign}. The Hadamard finite-part formula obtained from the parametrix in section \ref{sec:conformal} is explicit but does not make
the sign transparent. We instead deform the sector symmetrically by
\[
\Phi_\varepsilon(z)=z+\varepsilon e^{i(\pi-\alpha/2)}z^2,
\]
so that both one-sided limiting curvatures are
$2\varepsilon\sin(\alpha/2)$. For this deformation,
$c_{1/2}(\alpha)$ is the first variation of the local curved-corner coefficient
at $\varepsilon=0$. The Dirichlet domain variation formula expresses this
first variation as a boundary integral involving the square of the heat
Poisson kernel on an exact sector. After adding and subtracting the model half-plane contribution, the
coefficient becomes
\[
c_{1/2}(\alpha)
=
2\sin(\alpha/2)\int_0^\infty s^2
\left(\mathsf B_{\HH}(1,q_s)-\mathsf B_{W_\alpha}(1,q_s)\right)\,ds.
\]
Domain monotonicity for the Dirichlet heat kernel, together with the parabolic
Hopf lemma at the common boundary ray, then gives
\[
\operatorname{sgn} c_{1/2}(\alpha)=\operatorname{sgn}(\pi-\alpha).
\]
We also deduce the formula \eqref{eq:I-pi-over-N} in the case $\alpha=\pi/N$.

\section*{Acknowledgements}
S.\,L. was supported by a Taussky--Todd Fellowship and by NSF grant DMS-2346799. D.\,S. is grateful for support from the AMS-Simons research enhancement grant 501949-9208. He would also like to thank the Isaac Newton Institute for Mathematical Sciences, Cambridge, for support and hospitality during the programme Geometric spectral theory and applications, where work on this paper was undertaken. This work was supported by EPSRC grant EP/Z000580/1.

We thank Javier G\'omez-Serrano for suggesting the problem to us and Antoine Song for feedback on the paper.

\section{The corner contribution}\label{sec:gm}

Our approach to understanding the heat kernel on $\Omega$ is based on the methods of geometric microlocal analysis \cite{MelroseAPS, NRS}. The strategy is to first resolve the singularities of the heat kernel at $t=0$ by constructing a ``double heat space,'' $\Omega_h^2$, via a sequence of geometric blow-ups. The heat kernel lifts under these blow-ups to a well-behaved (``polyhomogeneous conormal'') function on the manifold with corners $\Omega_h^2$ \cite{NRS}. Roughly speaking, polyhomogeneous conormal functions are those whose asymptotic expansions at the boundary hypersurfaces are compatible and jointly polyhomogeneous at higher-codimension corners; the precise definition may be found in \cite{MelroseAPS}.

The boundary of $\Omega_h^2$ consists of boundary hypersurfaces corresponding to distinct asymptotic regimes. Among these hypersurfaces are:
\begin{itemize}
    \item The interior diagonal (td), modeling heat flow far from the boundary.
    \item The side faces (sf), modeling heat flow near a smooth edge.
    \item The front faces (ff), modeling heat flow at a corner.
\end{itemize}

One can, via geometric microlocal analysis, construct a parametrix $K(t,z,z')$---an approximate solution to the heat equation---by specifying its leading order behavior at each of the three boundary hypersurfaces td, sf, and ff. This should be done in such a way so that when the heat operator is applied to the parametrix $K(t,z,z')$, the result $R(t,z,z')$ is lower order at td, sf, and ff. In order to arrange this, we choose leading order behavior that solves the relevant model problems for the heat equation at each boundary hypersurface. We also must ensure that the specified behaviors at td, sf, and ff are consistent with each other at the intersections of the boundary hypersurfaces, so that there actually exists a polyhomogeneous conormal function on $\Omega_h^2$ with all of the correct behavior. Once we have a good enough parametrix, we can recover the true heat kernel via the Volterra series
\[H^D_{\Omega} = K - K\star R + K\star R\star R - \cdots.\]
To analyze this Volterra series, we use a composition formula for polyhomogeneous conormal functions on $\Omega_h^2$, which in our setting is \cite[Theorem 3.16]{NRS}. 

To be specific about what we mean when we talk about orders: set $T=\sqrt t$. We measure orders at each of td, sf, and ff in powers of $T$; a function which has leading order behavior $T^{-2}$ has order $-2$. ``Lower order'' means larger powers of $T$. Throughout, ``vanishes to infinite
order at \textrm{td} (or \textrm{sf}, or \textrm{ff})'' means the expansion has no terms of any finite $T$-order there.

The model cone heat kernel has order $-2$ at td, sf, and ff. We write
$\operatorname{ord}_{\mathrm{ff}}(\cdot)$ for this $T$-order. 
The key result of \cite[Theorem 3.16]{NRS} for our setting is that if $B$ vanishes to infinite order at td, then
$\operatorname{ord}_{\mathrm{ff}}(A\star B)=\operatorname{ord}_{\mathrm{ff}}(A)
+\operatorname{ord}_{\mathrm{ff}}(B)+4$.

As a first use of these techniques, we show that the corner contribution $\mathcal C_{1/2}(\alpha,\kappa_{+},\kappa_{-})$ indeed depends only on $\alpha$ and $\kappa_{\pm}$. Suppose that there exist two domains $\Omega$ and $\Omega'$, each with one corner, with the same $\alpha$ and the same $\kappa_{\pm}$. By \cite{NRS}, the heat kernels $H^D_{\Omega}$ and $H^D_{\Omega'}$ are polyhomogeneous conormal on $\Omega_h^2$ and $(\Omega')_h^2$ respectively. Since $\alpha$ is the same, the faces ff of these respective heat spaces are geometrically identical. We now state a lemma which informally says that: under these conditions, $H^D_{\Omega}-H^D_{\Omega'}$ has leading order at worst $t^0$ at ff.

\begin{lemma}\label{lem:local-corner}
Let $\Omega$ and $\Omega'$ be domains with a single corner having the same opening angle
$\alpha$ and the same one-sided boundary curvatures $\kappa_\pm$ at that corner.
Then $H^D_\Omega-H^D_{\Omega'}$ has order $\le T^{0}$ (at worst $T^0 = t^0$) at ff.
Equivalently, their ff coefficients of orders $T^{-2} = t^{-1}$ and $T^{-1} = t^{-1/2}$ agree.
\end{lemma}

\begin{proof} We use $H^D_{\Omega'}$ to build a parametrix $K$ for $H^D_{\Omega}$. We will build this parametrix $K$ for $H^D_\Omega$ by prescribing its leading behavior at the three faces. Specifically:
\begin{itemize}
\item At td, we specify that $K$ should have the usual interior expansion, which in this setting is simply the Euclidean heat kernel on $\mathbb R^2$, which has order $-2$ with no lower order terms.
\item 
At \textrm{sf}, we specify that $K$ should have the same expansion as the Dirichlet heat kernel on a smooth domain with boundary whose boundary coincides with the appropriate smooth part of $\partial\Omega$. Such a heat kernel is known to have a polyhomogeneous conormal expansion by \cite{Grieser}, and that expansion is consistent with the Euclidean heat kernel at \textrm{td}. The first term, at order $-2$, is universal. The second term, at order $-1$, follows from the boundary calculus construction in \cite{Grieser}; by a standard scaling argument, this coefficient is a universal multiple of the curvature $\kappa(z')$.
\item At ff, we specify that $K$ should have the same $t^{-1}$ (order $-2$) and $t^{-1/2}$ (order $-1$) terms as $H^D_{\Omega'}$. These are consistent with the interior expansion at td, since they are consistent for $H^D_{\Omega'}$.
\end{itemize}
We must check that the ff specifications are consistent with the sf ones. But this is the point of having $\kappa_{\pm}$ equal. Since the heat kernel $H^D_{\Omega'}$ is polyhomogeneous conormal, our ff specifications are consistent with the Euclidean behavior at td and also with the first two terms at sf on $(\Omega')_h^2$. But since $\kappa_{\pm}$ are identical between $\Omega$ and $\Omega'$, these terms on $(\Omega')_h^2$ have the same restrictions to sf$\cap$ff as the first two terms at sf on $\Omega_h^2$. Thus our ff and sf specifications are consistent. 

This consistency means that there exists a parametrix $K(t,z,z')$ with the specified expansions at td, sf, and ff. This parametrix solves the model problems to infinite order at td and sf and to two orders at ff. So $R(t,z,z')=(\partial_t-\Delta_{\Omega})K(t,z,z')$ has leading order $-2$, as opposed to $-4$, at ff, and leading order $\infty$ at td and sf. Using the composition formula of \cite{NRS}, we conclude that for each $n\ge 1$, $K\star(R^{\star n})$ has leading order at worst $t^0$ at ff. 
\[
  \operatorname{ord}_{\mathrm{ff}}(K\star R^{\star n})
  \ge (-2) + n(-2) + 4n = 2n-2 \ge 0
  \quad (n\ge1).
\]

Thus $$H^D_\Omega-K=\sum_{n\ge1}(-1)^n K\star R^{\star n}$$ has order at worst $0$ at ff, and since the same is true for $H^D_{\Omega'}-K$ by construction of $K$, so $H^D_\Omega-H^D_{\Omega'}=O(T^{0})$ at ff. The Lemma is proven. \end{proof}

We can now prove Theorem \ref{thm:thm1}.
\begin{proof}[Proof of Theorem \ref{thm:thm1}] As in \cite{NRS}, the heat trace is obtained by restricting $H^{D}_{\Omega}$ to the spatial diagonal and then integrating over $z\in\Omega$. The technical tool used is Melrose's pushforward theorem. We will not repeat all the details here, but the upshot is that each face td, sf, and ff gives a separate contribution to the expansion \eqref{eq:heattrace}. A priori, the expansions can interact to give logarithmic terms, but for the same reasons as in \cite[p. 49]{NRS}, keeping in mind that we now have two orders at ff rather than one, there are no logarithmic terms until at least $O(T^2\log T)=O(t\log t)$ (and potentially much later, if at all). The contribution from td is 
\[\frac{|\Omega|}{4\pi t} + O(t^{\infty}).\]
The contributions from sf are
\[-\frac{|\partial\Omega|}{8\sqrt{\pi t}} + \frac{1}{12\pi}\int_{\partial\Omega}\kappa\, ds +\frac{\sqrt t}{256\sqrt\pi}\int_{\partial\Omega}\kappa^2\, ds + O(t).\]
And the contributions from ff, since the first is already known and the second depends only on $\alpha$ and $\kappa_{\pm}$, are
\[\sum_{j=1}^{n}\frac{\pi^2-\alpha_j^2}{24\pi\alpha_j} + \sqrt t\sum_{j=1}^{n}\mathcal C_{1/2}(\alpha_j,\kappa_{j,+},\kappa_{j,-}) + O(t),\]
for some unknown function $\mathcal C_{1/2}(\alpha,\kappa_{+},\kappa_{-})$. Summing the three faces gives Theorem \ref{thm:thm1}.
\end{proof}
Everything in these proofs works in the Neumann case as well.

Note that we will perform a similar parametrix construction, in detail, in the next section, when we build a conformal model to find an expression for $\mathcal C_{1/2}(\alpha,\kappa_{+},\kappa_{-})$.

\section{The conformal model and the construction}\label{sec:conformal}

The goal of this section is to find an expression for the contribution $\mathcal C_{1/2}$. Since it depends only on $\alpha$ and on $\kappa_{\pm}$, we can work with a specific domain with those values. We take advantage of a conformal transformation to build such a domain, and then we construct its heat kernel.

\subsection{The conformal model domain}

We define a conformal mapping from $\mathbb C$ to $\mathbb C$. Let $w$ be the complex coordinate on the target space and $z$ be the complex coordinate on the source space. The Euclidean metric on the target $\mathbb{C}$ is $g_{Euc} = dw \, d\bar{w}$.
Fix $z_0\in\mathbb C$ and let $\Phi=\Phi_{z_0}: \mathbb{C} \to \mathbb{C}$ be the holomorphic map given by $$w = \Phi(z) = z+z_0z^2.$$ Near the origin, it is a small perturbation of the identity map.

To compute the pullback metric $\Phi^*(g_{Euc})$, we use the standard fact that for a holomorphic map $\Phi$, $dw = \Phi'(z) dz$ and $d\bar{w} = \overline{\Phi'(z)} d\bar{z}$.
Therefore,
\begin{align*} \Phi^*(g_{Euc}) &= \Phi^*(dw \, d\bar{w}) \\ &= (\Phi'(z) dz) (\overline{\Phi'(z)} d\bar{z}) \\ &= \Phi'(z) \overline{\Phi'(z)} dz \, d\bar{z} \\ &= |\Phi'(z)|^2 dz \, d\bar{z} \end{align*}
Thus, the pullback metric is:
$$ \Phi^*(g_{Euc}) = |1+2z_0z|^2 dz \, d\bar{z}. $$
Let $F(z) = |1+2z_0z|^2$ be the conformal factor derived above. An immediate calculation yields
\begin{equation}
F(z) = 1 + 4\Re(z_0z) + 4|z_0|^2|z|^2.
\end{equation}
If we write $z=x+iy$ and $z_0=u_0+iv_0$, we obtain
\begin{equation}
F(x+iy)=1 + 4(u_0x-v_0y) + 4(u_0^2+v_0^2)(x^2+y^2).
\end{equation}
On the other hand, if we write $z=re^{i\theta}$ and $z_0=r_0e^{i\theta_0}$, we obtain
\begin{equation}
F(re^{i\theta}) = 1 + 4r_0r\cos(\theta+\theta_0) + 4r_0^2r^2.
\end{equation}
We see that $F$, however it is written, is smooth, though of course not holomorphic, from $\mathbb C\to\mathbb R$. It is also smooth down to zero in polar coordinates $(r,\theta)$.

For later use, we also record the gradient of $F$ at the origin:
\begin{equation}
\nabla F(0,0)=\langle 4u_0,-4v_0\rangle = 4r_0\langle\cos\theta_0,-\sin\theta_0\rangle,
\end{equation}
where the latter expression is in polar coordinates but still measuring $x$ and $y$-derivatives.

We now revert to $\mathbb R^2$. Let $\Omega_0$ be the flat sector $\{z \in \mathbb{C} \mid 0 < \arg(z) < \alpha \}$. The map $\Phi$ takes $\Omega_0$ to a curvilinear domain $\Omega$ with the same corner angle $\alpha$. 
The signed curvatures of the two boundary edges of $\Omega$ at the corner are given by:
\begin{itemize}
    \item $\kappa_+$ (on the image of the $\theta=0$ ray): $2\Im(z_0)$
    \item $\kappa_-$ (on the image of the $\theta=\alpha$ ray): $-2\Im(z_0e^{i\alpha})$
\end{itemize}
Here the orientation is chosen so that the signed curvatures measure curvature in the inward-pointing direction, so that a convex domain will have $\kappa_+\ge 0$ and $\kappa_-\ge 0$. A domain with symmetry about $\arg(z)=\alpha/2$ will have $\kappa_+=\kappa_-$. 

By choosing the complex parameter $z_0$, we can match any desired curvatures $\kappa_+, \kappa_-$, as long as $\alpha\ne\pi$. Solve the following linear system for $z_0 = u_0 + iv_0$:
\begin{align*}
    \kappa_+ &= 2 v_0 \\
    \kappa_- &= -2 \Im((u_0+iv_0)(\cos\alpha +i\sin\alpha)) = -2(u_0\sin\alpha + v_0\cos\alpha)
\end{align*}
This yields $z_0=u_0+iv_0$, where
\begin{equation}
\begin{cases} u_0= & -\frac 12\kappa_+\cot\alpha - \frac 12\kappa_-\csc\alpha;\\ v_0 = & \frac 12\kappa_+.
\end{cases}
\end{equation}
In terms of the curvatures, we observe that
\begin{equation}\label{eq:defofr0}
r_0 = \frac{1}{2}|\csc\alpha| \sqrt{\kappa_+^2+2\kappa_+\kappa_-\cos\alpha+\kappa_-^2};\quad \theta_0=\textrm{arg}(u_0 + iv_0). 
\end{equation}
Equivalently,
\begin{equation}\label{eq:kappas-via-theta0}
\kappa_+=2r_0\sin\theta_0,
\qquad
\kappa_-=-2r_0\sin(\alpha+\theta_0).
\end{equation}
Hence
\begin{equation}\label{eq:bisector-amplitude}
r_0\cos\left(\theta_0+\frac{\alpha}{2}\right)
=-\frac{\kappa_+ + \kappa_-}{4\sin(\alpha/2)}.
\end{equation}

The pullback by $\Phi$ of the Euclidean metric on $\Omega$ is
\[ F \cdot g_{Eucl,\Omega_0},\]
where we now interpret $F$ as a function with domain $\mathbb R^2$ rather than $\mathbb C$.
By the formula for the Laplacian in two dimensions under a conformal change of metric $g = F g_0$, the new Laplacian is $\Delta_g = F^{-1} \Delta_{g_0}$. Therefore, the pullback Laplacian is
\[ \Phi^*(\Delta_{\Omega})=F^{-1}\Delta_{\Omega_0}. \]
We want to understand the heat kernel for the operator $\Delta_\Omega$ on the curved domain $\Omega$. By pulling back, this is equivalent to studying the heat kernel for the operator $F^{-1}\Delta_{\Omega_0}$ on $\Omega_0$. The corresponding heat equation is $\partial_t H = F^{-1}\Delta_{\Omega_0} H$. We define our heat operator on the model space $\Omega_0$ as:
\[ \mathcal L=\partial_t - F^{-1}\Delta_{\Omega_0}. \]

We now construct the heat kernel for $\mathcal L$. We use the techniques of \cite{NRS}, working at td, then at sf and ff. This corresponds to building a parametrix for the heat kernel, first at the interior diagonal, and then at the diagonal both near the boundary and near a corner. Once we have a sufficiently accurate parametrix, we iterate away the error via a Volterra series to produce the true heat kernel.

\subsection{The heat kernel at td}
It is proved in \cite{NRS} that the heat kernel for $\mathcal L$, being the pullback of the heat kernel for the Laplacian on the curvilinear polygon $\Omega$, is polyhomogeneous conormal on the double space of \cite{NRS}. We also know that its full expansion at td is the same as that for the heat kernel of the operator $\Phi^*(\Delta_{\mathbb R^2}) = F^{-1}\Delta_{\mathbb R^2}$. This expansion may be written in terms of the (interior) boundary defining function $T=\sqrt t$ as
\[T^{-2}a_{-2}+T^{-1}a_{-1}+T^0a_0+\ldots,\]
where $a_{-2}$, $a_{-1}$, $a_0$, etc. are smooth functions on the face td. 
We refer to $a_k$ as the td-coefficients of the heat kernel; in particular, $a_0$ denotes the coefficient of order $T^0$ at td, whose restriction along
$\mathrm{td}\cap\mathrm{ff}$ will determine $\mathcal{C}_{1/2}(\alpha,\kappa_+,\kappa_-)$.

From \cite{NRS}, these functions are smooth down to ff as well. The boundary defining function $T=\sqrt t$ vanishes to first order at ff (as at td). Thus a td term $T^k a_k$, viewed as a kernel, has order $k$ at ff.\footnote{The heat trace is obtained by restricting the kernel to the spatial diagonal and integrating against the Riemannian area element. Near ff, in polar coordinates about the vertex one has $r=TR$, so for fixed $t$ (hence fixed $T$) one gets $dr=T\,dR$ and \[ r\,dr\,d\theta = T^2\,R\,dR\,d\theta. \] Accordingly, the corresponding trace density carries an additional factor $T^2$, so the contribution of $T^k a_k$ to the trace integrand has order $k+2$ at ff. This is the bookkeeping in the pushforward argument at td$\cap$ff: the shift by $2$ rules out logarithmic terms at this intersection. Similar considerations apply at the other intersections.} Therefore, the only terms which are relevant to the leading and subleading terms at ff are $T^{-2}a_{-2}+T^{-1}a_{-1}$.

We now use the work of Grieser \cite{Grieser} to compute $a_{-2}$ and $a_{-1}$. We use his notation throughout this subsection: the coordinates are $(t,x,y)$, where $x\in\mathbb R^2$ is an interior coordinate on the first factor and $y\in\mathbb R^2$ on the second. The initial parametrix for the heat kernel is
\begin{equation}
    K_1(t,x,y)=\frac{1}{4\pi t}\exp\left(-\frac{F(y)|x-y|^2}{4t}\right).
\end{equation}
Computing the application of the heat operator $(\partial_t-F^{-1}(x)\Delta_{\mathbb R^2})$ to $K_1$, we obtain the error term
\begin{equation}
    R_1(t,x,y):=(\partial_t-F^{-1}(x)\Delta_{\mathbb R^2})K_1(t,x,y)=\left(1-\frac{F(y)}{F(x)}\right)\frac{\partial}{\partial t}K_1(t,x,y).
\end{equation}

\noindent   \textit{Notation (td symbols).}
In the rescaled variables $(t,X,y)$ with $X=(x-y)/\sqrt t$, any heat kernel $A$ will admit an asymptotic of the form
\begin{equation}\label{eq:asymptotic}A(t,x,y)\sim\sum_{q \in  d-\frac12\mathbb N_0} t^{-2-q}\,\Phi_q(A)(X,y)\end{equation} as $t\to0$ with $X$ fixed.
The coefficient $\Phi_q(A)$ is the \emph{td-symbol of order $q$} of the operator $A$; if $A$ has order $d$, then $\Phi_d(A)$ is the principal td-symbol of $A$. Referring to Grieser \cite{Grieser} for the definition of the various operator spaces $\Psi_H^d$, we have $K_1(t,x,y)\in\Psi_H^{-1}$. Roughly, this means that $K_1(t,x,y)$ has an expansion of the form \eqref{eq:asymptotic} with $d=-1$, hence leading term of order $t^{-1}$, and with coefficients decaying to infinite order at $|X|=\infty$.

The prefactor is $O(\sqrt t)$ near the diagonal. Since $K_1\in\Psi_H^{-1}$, one has $\partial_tK_1\in\Psi_H^{0}$, which has size $t^{-2}$, so
$R_1(t,y+\sqrt t\,X,y)$ has size $t^{-3/2}$ in td coordinates. Therefore $R_1\in\Psi_H^{-1/2}$ with principal td-symbol
\begin{equation}
    \Phi_{-3/2}(R_1)(X,y)=(\frac{X\cdot\nabla(F)(y)}{F(y)})(\frac{F(y)}{4}X^2-1)\frac{1}{4\pi}\exp(-\frac{F(y)}{4}X^2).
\end{equation}

By the heat kernel construction in Grieser \cite{Grieser}, the leading and subleading terms of the true heat kernel at td will be the same as those of
\[K_1-K_1*R_1.\]
$K_1$ has zero subleading term in the coordinates $(t,X,y)$, so for the subleading term we just need to understand $K_1*R_1$, in particular just the \emph{leading} term of $K_1*R_1$, which is $\Phi_{-3/2}(K_1*R_1)$. This can be computed using \cite[Proposition 2.6]{Grieser}. Observe that the integration in $Z$ is with respect to $F(y)\, dZ$, the metric at $y$. We obtain, including the minus sign,
\begin{multline}
    \Phi_{-3/2}(-K_1*R_1)(X,y)=-\int_0^1\int_{\mathbb R^2}(1-\sigma)^{-1}\sigma^{-3/2}\frac{1}{4\pi}\exp(-\frac{F(y)}{4}\frac{(X-Z)^2}{1-\sigma})\\
    \frac{Z\cdot\nabla(F)(y)}{\sqrt\sigma F(y)}(\frac{F(y)}{4}\frac{Z^2}{\sigma}-1)\frac{1}{4\pi}\exp(-\frac{F(y)}{4}\frac{Z^2}{\sigma})\, F(y) dZ\, d\sigma.
\end{multline}

This integral can be computed, which we now do. Let
\[c=\frac{F(y)}{4};\quad V=\nabla(F)(y);\quad \tilde Z=\frac{Z}{\sqrt{\sigma}};\quad \tilde X=\frac{X}{\sqrt{\sigma}}.\]
We obtain
\begin{multline}\label{integraltodo}
    \Phi_{-3/2}(-K_1*R_1)=-\frac{1}{16\pi^2}\int_0^1(1-\sigma)^{-1}\sigma^{-1/2}\\ \int_{\mathbb R^2}
    \tilde Z\cdot V(c\tilde Z^2-1)\exp(-c\tilde Z^2-c\sigma\frac{(\tilde X-\tilde Z)^2}{1-\sigma})\, d\tilde Z\, d\sigma.
\end{multline}
As usual, we can complete the square to do the inner integral. Returning to $X$, we get
\[\exp(-cX^2)\int_{\mathbb R^2}\tilde Z\cdot V(c\tilde Z^2-1)\exp(-\frac{c}{1-\sigma}(\tilde Z-\sqrt{\sigma}X)^2)\, d\tilde Z.\]
Let $\hat Z=\tilde Z-\sqrt{\sigma}X$:
\[\exp(-cX^2)\int_{\mathbb R^2}(\hat Z+\sqrt\sigma X)\cdot V(c(\hat Z+\sqrt\sigma X)^2-1)\exp(-\frac{c}{1-\sigma}\hat Z^2)\, d\hat Z.\]
Any of the terms in the prefactor
\[((\hat Z+\sqrt\sigma X)\cdot V)(c(\hat Z+\sqrt\sigma X)^2-1)\] 
which are odd in one of the coordinates of $\hat Z$ yield zero when integrated. Up to such terms, that pre-factor is
\[(\sqrt{\sigma}X\cdot V)(c\hat Z^2+c\sigma X^2-1)+2c\sqrt{\sigma}(\hat z_1^2x_1v_1+\hat z_2^2x_2v_2).\]
Using the usual formulas for Gaussian integrals, we get that the inner integral in \eqref{integraltodo} is
\[\exp(-cX^2)[(\sqrt\sigma X\cdot V)(c\sigma X^2-1)\frac{\pi(1-\sigma)}{c} + (\sqrt\sigma X\cdot V)\frac{2\pi(1-\sigma)^2}{c}]\]
\[=c^{-1}\pi\sigma^{1/2}(1-\sigma)\exp(-cX^2)(X\cdot V)(c\sigma X^2+1-2\sigma),\]
and therefore
\begin{equation}
    \Phi_{-3/2}(-K_1*R_1)=-\frac{1}{16c\pi}\exp(-cX^2)(X\cdot V)\int_0^1(c\sigma X^2+1-2\sigma)\, d\sigma,
\end{equation}
Performing the simple integration, we obtain
\begin{equation}
    \Phi_{-3/2}(-K_1*R_1)=-\frac{1}{32\pi}\exp(-cX^2)(X\cdot V)X^2.
\end{equation}
Inserting the variables, we have
\begin{equation}\label{eq:Phi_minus}
    \Phi_{-3/2}(-K_1*R_1)=-\frac{1}{32\pi}X^2(X\cdot\nabla F(y))\exp(-\frac{F(y)}{4}X^2).
\end{equation}

We can now write down the two leading orders of the heat kernel at td: they are

\begin{equation}\label{eq:kernel_at_td}
T^{-2}\frac{1}{4\pi}\exp(-\frac{F(y)}{4}X^2)- T^{-1}\frac{1}{32\pi}X^2(X\cdot\nabla F(y))\exp(-\frac{F(y)}{4}X^2).
\end{equation}
As one may check, this expression is in the kernel of $T^2\mathcal L$ up to $O(T^0)$.

\begin{remark} Two observations:
\begin{itemize}
    \item If $F$ is constant, the subleading term is zero. In fact all subleading terms are zero -- this is the Euclidean setting, where the initial parametrix is equal to the true heat kernel. %
    \item In all cases, $\Phi_{-3/2}(K_1*R_1)(0,y)=0$. This means that the restriction of $K_1*R_1$ to the spatial diagonal has no term of order $T^{-1}$ at td. This is unsurprising, as we do not expect any term of the form $T^{-1}=t^{-1/2}$ in the interior heat trace expansion.
\end{itemize}
\end{remark}

\subsection{A parametrix for the heat kernel}
We take our heat operator $\mathcal L$ and apply it to a parametrix
\[H^{(1)}:=G\cdot H_{cone},\]
where $H_{cone}$ is the Dirichlet heat kernel for an exact sector of angle $\alpha$, and where $G$ is a scalar function chosen as a corrector to allow $H^{(1)}$ to satisfy the heat equation to two orders at td. 

\noindent\textit{Notation.}
In this subsection we set $T=\sqrt t$ and use the rescaled difference vector
\[
X:=\frac{z-z'}{T}\in\mathbb R^2,
\qquad |X|=\frac{|z-z'|}{T}.
\]
We write $X\cdot\nabla F(z')$ for the Euclidean dot product and $|X|^2=X\cdot X$.

\medskip
Specifically, define
\begin{equation}\label{eq:G-def}
G:=\Bigl(1-\frac18\,T\,|X|^2\,(X\cdot\nabla F(z'))\Bigr)
\exp\Bigl[-\frac{F(z')-1}{4}\,|X|^2\Bigr].
\end{equation}
Then $G$ times the Euclidean heat kernel (the td model for $H_{\cone}$) yields
\eqref{eq:kernel_at_td}.
We do not claim that $G$ is polyhomogeneous on the double space; if $F(z')<1$ then the factor
$\exp \bigl(\frac{1-F(z')}{4}|X|^2\bigr)$ grows as $|X|\to\infty$ (note that $|X|=|z-z'|/T$).
However, $G$ is polyhomogeneous away from $\mathrm{td}$, and $G\,H_{cone}$ is polyhomogeneous on the
double space. Moreover, $G\,H_{cone}$ satisfies Dirichlet boundary conditions\footnote{The Neumann problem is harder precisely because such a parametrix would not necessarily inherit Neumann boundary conditions. This can likely be worked around but we do not do so here.}, inheriting this from
$H_{cone}$.

Now examine $G$ at ff, the corner front face. This face is the most singular stratum of the heat space, where $t \to 0$ and both points $x,y=0$. The function $G$ is equal to 1 at ff because $T=0$ at ff, and $F=1$ at ff. As a consequence, $H^{(1)}$ is in the kernel of the heat operator $\mathcal L$ to leading order at ff; we say that it ``solves the model problem'' to one order at ff. It also solves the model problem to two orders at td. And since $T=0$ at sf and $\exp[-\frac 14(F(y)-1)X^2]H_{cone}$ agrees to first order at sf with the initial parametrix $\exp[-\frac 14(F(y)-1)X^2]H_{\mathbb R^2_+}$ used by Grieser at sf \cite{Grieser}, $H^{(1)}$ also solves the model problem to one order at sf.

Now define:
\[R^{(1)} := \mathcal L(H^{(1)}).\]
If $H^{(1)}$ were not correctly chosen, we would expect $R^{(1)}$ to be two orders worse than $H^{(1)}$ at each of td, ff, and sf, which means order $-4$ at each. However,
since $H^{(1)}$ solves the model problem to two orders at td and one order at ff, and also one order at sf, $R^{(1)}$ has leading orders given by $-2$ at td, $-3$ at ff, and $-3$ at sf. 

In order to use the composition theorem of \cite[Theorem 3.16]{NRS}, which requires one of the two operators to vanish to infinite order at td, we work with a slightly improved parametrix $H^{(2)}$ given by solving away the full expansion at td. This equals $H^{(1)}$ plus a term which is order \emph{zero} at td, ff, and sf.   
Let $D:=H^{(2)}-H^{(1)}$. Then $D$ has order $\ge 0$ at ff. Hence if we use the notation
\[R^{(2)}:=\mathcal L(H^{(2)}),\]
then 
$R^{(2)}-R^{(1)}=\mathcal L(D)$ has order $\ge -2$ at ff and sf.
Moreover, $R^{(2)}$ vanishes to infinite order at td and it has leading order $-3$ at both ff and sf.

By Duhamel's principle and our usual Neumann series sum, the true heat kernel for our conformal metric is
\[H=H^{(2)} - H^{(2)}*R^{(2)}+\sum_{k=2}^{\infty}(-1)^kH^{(2)}*(R^{(2)})^{k}.\]
We are interested in finding its term of order $-1$ at ff (the sub-leading term). By the composition formula in \cite{NRS}, using in particular that
\[\text{ord}_{\text{ff}}(A*B) = \text{ord}_{\text{ff}}(A) + \text{ord}_{\text{ff}}(B) + 4,\]
the terms with $k\ge 2$ have no order $-1$ term at ff, so the order $-1$ term of our true heat kernel at ff is the same as that of
\[H^{(2)}-H^{(2)}*R^{(2)}.\]
This in turn\footnote{The justification for using the simpler parametrix $H^{(1)}$ to compute the subleading term at the {ff} face is done by an analysis of operator orders within the composition calculus of \cite{NRS}. Let $D = H^{(2)} - H^{(1)}$ be the correction term that improves the parametrix at the {td} face. By construction, the leading order of $D$ at the {ff} face is at least $T^0$. The difference between the first two terms of the respective Neumann series is $\Delta H = D - H^{(1)}*\mathcal{L}(D) - D*R^{(1)} - D*\mathcal{L}(D)$. A direct application of the composition formula shows that each term in this expression has a leading order of $T^0$ or higher at the {ff} face. For instance, the most singular term, $H^{(1)}*\mathcal{L}(D)$, has an order of at least $-2 + (-2) + 4 = 0$. Since the entire difference $\Delta H$ is no more singular than $T^0$, the expressions $H^{(1)}-H^{(1)}*R^{(1)}$ and $H^{(2)}-H^{(2)}*R^{(2)}$ must agree for all terms of order $T^{-1}$ and below. We are therefore justified in using the simpler expression to compute the desired subleading coefficient.} is the same as that of
\[H^{(1)}-H^{(1)}*R^{(1)}=GH_{cone}-GH_{cone}*(\mathcal L(GH_{cone})).\]
Here $H^{(1)}$ has order $-2$ at ff, so we need its leading and subleading terms. $R^{(1)}$ has leading ff order $-3$. Thus the leading ff term of the convolution $H^{(1)}*R^{(1)}$ comes from $T^{-2} \ast T^{-3}$, which lands at $-1$ at ff, by the ff composition rule.

A critical fact here is that by the composition formula in \cite{NRS}, the leading term of $H^{(1)}*R^{(1)}$ at ff depends only on the leading terms of $H^{(1)}$ and $R^{(1)}$ at ff, not on their leading terms at other boundary faces (in particular at sf). This is helpful for the following reasons. 
\begin{itemize}
    \item First, the integral in the convolution $H^{(1)}*R^{(1)}$ is done with respect to the metric $Fg_{Eucl}$, but since we are computing at the ff face, which corresponds to the vertex where $r=0$, we can evaluate $F(z)$ at that point. As established, $F(0)=1$. So, for the purpose of calculating the leading term of the convolution, the integration measure simplifies from the Riemannian volume form to the standard Euclidean one.
    \item Second, when computing the second convolution factor $\mathcal L(GH_{cone})$, we can ignore any term below its leading order $-3$. Similarly, for the first convolution factor, $GH_{\text{cone}}$, we only need its leading order term, which is of order $T^{-2}$. All smoother terms can be ignored as they will not affect the leading term of the result.
\end{itemize} 

This leads to the following simplification. Consider $GH_{cone}*(\mathcal L(GH_{cone}))$. We only need to worry about the leading order of each term, so $G$ may be replaced with 1 and we need only analyze
\[H_{cone}*(\mathcal L(GH_{cone})).\]
Let us compute $\mathcal L(GH_{cone})$. It is
\begin{equation}
    \mathcal L(GH_{cone})=\mathcal L(G)H_{cone} + G\mathcal L(H_{cone}) - 2F^{-1}\nabla G\cdot\nabla(H_{cone}).
\end{equation}
Using the fact that $\mathcal L_{cone}H_{cone}=0,$ this becomes
\begin{equation}\mathcal L(GH_{cone})=\mathcal L(G)H_{cone} - G(F^{-1}-1) \Delta(H_{cone}) - 2F^{-1}\nabla G\cdot\nabla(H_{cone}).\end{equation}
We may replace the $\mathcal L$ with $\mathcal L_{cone}$ in the first term, as $(\mathcal L-\mathcal L_{cone})(G)$ has order zero at ff. We may also drop the $G$ in the second term and the $F^{-1}$ in the third term, as they are both 1 to leading order at ff. Thus up to terms which are irrelevant for our calculation,
\[\mathcal L(GH_{cone})\sim \mathcal L_{cone}(G)H_{cone}-(F^{-1}-1)\Delta(H_{cone})-2\nabla G\cdot\nabla(H_{cone}).\]
In conclusion, to two orders at ff, we have
\begin{equation}\label{eq:realhkapproxatff}
H\sim GH_{cone} - H_{cone}*(\mathcal L_{cone}(G)H_{cone} -(F^{-1}-1)\Delta(H_{cone})- 2\nabla G\cdot\nabla H_{cone}),
 \end{equation}
where the convolution -- and all derivatives here -- are taken with respect to the Euclidean metric on the cone. The geometric correction factor $G$ must be retained in the stand-alone factor $G\,H_{\mathrm{cone}}$ (before taking the trace): its $T^{+1}$ contribution times the $T^{-2}$ of $H_{\mathrm{cone}}$ produces one of the two $T^{-1}$ pieces in the kernel at ff we are calculating.

\subsection{Computations with $G$}

We now turn to the explicit calculation of the terms in Equation \eqref{eq:realhkapproxatff}. We are interested in the first two terms in their expansion at ff.

Throughout, we replace Grieser's $x$ and $y$ by $z$ and $z'$. Let $z$ have polar coordinates $(r,\theta)$ and let $z'$ have polar coordinates $(r',\theta')$, and then let 
\[R=r/T, R'=r'/T.\]
With this notation, natural local coordinates near the interior of the front face ff are $(T,R,\theta,R',\theta')$, where $T$ is a boundary defining function. In these coordinates, we compute the first two terms in the expansion \eqref{eq:realhkapproxatff}, of orders $T^{-2}$ and $T^{-1}$. 
To do this we compute two terms in the expansion of $GH_{cone}$. We also compute the leading $T^{-2}$ term in the expansion of $H_{cone}$ and the leading $T^{-3}$ term in the expansion of the second convolution factor, as by the composition formula in \cite{NRS} their convolution will give a term of order $T^{-2-3+4}=T^{-1}$ at ff. To this end, we denote
\begin{equation}
    a(R,\theta,R',\theta')=T^2H_{cone}|_{\mathrm{ff}};
    \end{equation}
    \begin{equation}\label{eq:binitial}b(R,\theta,R',\theta')= r_0^{-1}(T^3(\mathcal L_{cone}(G)H_{cone} -(F^{-1}-1)\Delta(H_{cone})- 2\nabla G\cdot\nabla H_{cone}))|_{\mathrm{ff}}.
\end{equation}
Our goal is to find expressions for $a$ and $b$.

Let $\varphi_j(\theta)$ be the $L^2$-normalized eigenfunction
\[\varphi_j(\theta)=\sqrt{\frac{2}{\alpha}}\sin(\frac{\pi j\theta}{\alpha}).\]
First, we use Cheeger's formula for the Dirichlet heat kernel on the exact sector of angle $\alpha$ \cite{Cheeger1983}.
In polar coordinates $z=(r,\theta)$ and $z'=(r',\theta')$, one has
\[
H_{\mathrm{cone}}(t; r,\theta,r',\theta')
=\frac{1}{2t}\exp\Big[-\frac{1}{4t}\big(r^{2}+(r')^{2}\big)\Big]
\sum_{j=1}^{\infty} I_{\mu_j}\Big(\frac{r r'}{2t}\Big)\,\varphi_j(\theta)\,\varphi_j(\theta').
\]
Now set $T=\sqrt t$, $R=r/T$, and $R'=r'/T$. Since $r=TR$, $r'=TR'$, and $t=T^2$, this becomes
$H_{\mathrm{cone}}=T^{-2}a(R,\theta,R',\theta')$ by a simple scaling argument, with $a$ given by
\begin{equation}\label{eq:a-def}
    a(R,\theta,R',\theta')
    =\frac12\,\exp \Big[-\frac{R^2+(R')^2}{4}\Big]
    \sum_{j=1}^\infty I_{\mu_j} \Big(\frac{RR'}{2}\Big)\,
    \varphi_j(\theta)\,\varphi_j(\theta').
\end{equation}
To be absolutely clear, we write
\[
H_{\mathrm{cone}}(t; r,\theta,r',\theta')
= T^{-2}\,
a \left(\frac{r}{\sqrt t},\theta,\frac{r'}{\sqrt t},\theta'\right),
\qquad T=\sqrt t,
\]

For all other terms we must analyze $G$.
Recall that $G$ is a function of $T=\sqrt{t}$, the source point $z'$, and the scaled, relative position vector $X = (z-z')/T$. We have the relations
\begin{align*}
    z &= (r\cos\theta, r\sin\theta) = (RT\cos\theta, RT\sin\theta) \\
    z' &= (r'\cos\theta', r'\sin\theta') = (R'T\cos\theta', R'T\sin\theta').
\end{align*}
We can directly compute the components of the scaled vector $X$:
\begin{equation}
    X = (R\cos\theta - R'\cos\theta', R\sin\theta - R'\sin\theta').
\end{equation}
Define the distance function $D=D(R,\theta,R',\theta')$ by
\begin{equation}
    D^2:=|X|^2 = (R\cos\theta - R'\cos\theta')^2 + (R\sin\theta - R'\sin\theta')^2 = R^2 + (R')^2 - 2RR'\cos(\theta-\theta').
\end{equation}
For future use we also define
\begin{equation}\label{eq:S-def}
S=S(R,\theta,R',\theta')= \left( R\cos(\theta+\theta_0) + R'\cos(\theta'+\theta_0) \right).
\end{equation}
With these relations, we can write the full expression for $G$ in the front-face coordinates:
\begin{multline}
    G(T,R,\theta,R',\theta')=\Bigg[1-\frac 18TD^2\langle R\cos\theta-R'\cos\theta',R\sin\theta-R'\sin\theta'\rangle\cdot\nabla F(R'T,\theta')\Bigg]\\\exp[-\frac 14(F(R'T,\theta')-1)D^2].   
\end{multline}

Now we compute the expansion of $G$ at ff in powers of $T$. Since $F$ is smooth, we have by Taylor series that
\begin{equation*}
	G \approx 1 - \frac{1}{8}TD^2 \langle R\cos\theta - R'\cos\theta', R\sin\theta - R'\sin\theta' \rangle \cdot \nabla F(0,0) - \frac{1}{4}TD^2 \langle R'\cos\theta', R'\sin\theta' \rangle \cdot \nabla F(0,0) + O(T^2).
\end{equation*}
By combining the vector coefficients and plugging in $\nabla F(0,0)=4r_0\langle\cos\theta_0,-\sin\theta_0\rangle$, this expression simplifies:
\begin{align}\label{eq:gsimplifying}
	G &= 1 - TD^2 \left[ \frac{1}{8}\langle R\cos\theta + R'\cos\theta', R\sin\theta + R'\sin\theta' \rangle \right] \cdot \left( 4r_0\langle\cos\theta_0,-\sin\theta_0\rangle \right) + O(T^2) \\
	&= 1 - \frac{r_0}{2} TD^2 \langle R\cos\theta + R'\cos\theta', R\sin\theta + R'\sin\theta' \rangle \cdot \langle\cos\theta_0,-\sin\theta_0\rangle + O(T^2)\\
    &= 1 - \frac{r_0}{2} TD^2 \left( R(\cos\theta\cos\theta_0-\sin\theta\sin\theta_0) + R'(\cos\theta'\cos\theta_0-\sin\theta'\sin\theta_0) \right) + O(T^2) \\
	&= 1 - \frac{r_0}{2} TD^2S + O(T^2).
\end{align}
From this we immediately conclude that, at ff,
\begin{equation}\label{eq:ghconeatff}
    GH_{cone} = T^{-2}a(R,\theta,R',\theta') - \frac {r_0}{2}T^{-1}D^2Sa(R,\theta,R',\theta') + O(1).
\end{equation}
One nice feature of this expression is that, since $D^2$ is zero on the spatial diagonal, the sub-leading term of $GH_{cone}$ gives no contribution to the heat trace.

The last step is to find an expression for $b(R,\theta,R',\theta')$ by simplifying \eqref{eq:binitial}. As a first step, we examine the middle term in \eqref{eq:binitial}. We have
\[F^{-1}(r,\theta)=1-4r_0r\cos(\theta+\theta_0)+O(r^2),\]
so in our front face coordinates,
\[(F^{-1}-1)=-4r_0TR\cos(\theta+\theta_0)+O(T^2).\]
We use the notation
\[\Delta_{R,\theta} = \partial_R^2 + \frac{1}{R}\partial_R + \frac{1}{R^2}\partial_\theta^2, \qquad \nabla_{R,\theta} = (\partial_R, \frac{1}{R}\partial_\theta).\] 
With this notation, we have $\nabla_{r,\theta} = T^{-1} \nabla_{R,\theta}$ and $\Delta_{r,\theta} = T^{-2} \Delta_{R,\theta}$, and so
\[\Delta H_{cone}=T^{-2}\Delta_{R,\theta}H_{cone},\] which by conic scaling is \[T^{-4}\Delta_{R,\theta}a(R,\theta,R',\theta').\]
Therefore
\[-(F^{-1}-1)\Delta H_{cone} = 4r_0T^{-3}R\cos(\theta+\theta_0)\Delta_{R,\theta}a(R,\theta,R',\theta'),\]
and $T^3r_0^{-1}$ times this is the contribution that is made to $b(R,\theta,R',\theta')$:
\begin{equation}
    -T^3r_0^{-1}(F^{-1}-1)\Delta H_{cone}=4R\cos(\theta+\theta_0)\Delta_{R,\theta}a.
\end{equation}

The other terms in \eqref{eq:binitial} involve derivatives of $G$, so we must compute these using \eqref{eq:gsimplifying}. At ff, 
\[\nabla G=-\frac{r_0}{2}\nabla_{R,\theta}(D^2S) + O(T);\quad \Delta G = -\frac{r_0}{2}T^{-1}\Delta_{R,\theta}(D^2S) + O(1).\]
On the other hand, $\partial_t$ lifts to
\[\frac 12T^{-2}(T\partial_T-R\partial_R-R'\partial_{R'})\]
and so
\begin{equation}\partial_tG =-\frac{r_0}{4}T^{-1}D^2S+\frac{r_0}{4}T^{-1} (R\partial_R+R'\partial_R')(D^2S)+ O(1).\end{equation}
All in all,
\begin{equation}\mathcal LG = -\frac{r_0}{4}T^{-1}(1-R\partial_{R}-R'\partial_{R'}-2\Delta_{R,\theta})(D^2S) + O(1).\end{equation}
\begin{lemma}\label{lem:polarops}
With $\Delta_{R,\theta}$ and $\nabla_{R,\theta}$ the polar Laplacian/gradient in $(R,\theta)$,
\begin{equation*}
\Delta_{R,\theta}(D^2)=4,\qquad \Delta_{R,\theta}S=0,\qquad \Delta_{R,\theta}(D^2 S)=8R\cos(\theta+\theta_0),\qquad
(R\partial_R+R'\partial_{R'})(D^2 S)=3D^2 S.
\end{equation*}
\end{lemma}
\begin{proof}
    This is a direct computation, using for instance that $\Delta_{R,\theta}R^2=4$ and $S$ is harmonic in $(R,\theta)$.
\end{proof}
Putting all of this together, we obtain
\begin{multline}\label{eq:bexpression}
b(R,\theta,R',\theta')=\left(\frac 12D^2S+4R\cos(\theta+\theta_0)\right)a(R,\theta,R',\theta')\\+4R\cos(\theta+\theta_0)\Delta_{R,\theta}a(R,\theta,R',\theta')\\ +\nabla_{R,\theta}(D^2S)\cdot\nabla_{R,\theta}a(R,\theta,R',\theta').
\end{multline}

For later use, we rewrite the gradient term by using the identity
\begin{equation}\label{eq:gradidentity}
    \nabla u\cdot\nabla v = \frac 12\Delta(uv) - \frac 12v\Delta u -\frac 12u\Delta v.
\end{equation}
This yields, after some computations,

\begin{equation}\label{eq:bexpression-clean}
    \begin{aligned}
    b(R,\theta,R',\theta')&=
    \Big(\tfrac12 D^2S\Big)\,a(R,\theta,R',\theta')\\
    &\quad+\,4 R\cos(\theta+\theta_0)\,\Delta_{R,\theta}a(R,\theta,R',\theta')\\
    &\quad+\,\tfrac{1}{2}\,\Delta_{R,\theta}\big(D^2S\cdot a(R,\theta,R',\theta')\big)
    -\tfrac{1}{2}\,(D^2S)\,\Delta_{R,\theta}a(R,\theta,R',\theta').
    \end{aligned}
    \end{equation}
Either \eqref{eq:bexpression} or \eqref{eq:bexpression-clean} may be used.

\subsubsection{The little trace and the convolution}
To find the heat trace, we first restrict our kernel to the diagonal to find the so-called ``little trace". By \eqref{eq:ghconeatff}, the restriction to the diagonal of the $T^{-1}$ coefficient of $GH_{cone}$ is zero, and therefore the first term in \eqref{eq:realhkapproxatff} does not contribute to the $T^{-1}$ term of interest. The only contribution comes from the second term in \eqref{eq:realhkapproxatff}, and that contribution is the convolution of the leading orders of the two factors. This convolution is
\[-\int_0^t\int_0^{\infty}\int_0^{\alpha} (t-s)^{-1}a(\frac{r}{\sqrt{t-s}},\theta,\frac{r'}{\sqrt{t-s}},\theta')s^{-3/2}b(\frac{r'}{\sqrt{s}},\theta',\frac{r''}{\sqrt{s}},\theta'')\, r'\, d\theta'\, dr'\, ds.\]
Switching to front face variables yields that the term of interest, $H_{\text{sub}}$, is the following: here $\cong$ denotes having the same heat trace coefficient at order $T^{-1}$,
\begin{multline}
    H_{\text{sub}} \cong -T^{-1} \int_0^1 \int_0^{\infty} \int_0^{\alpha} (1-\sigma)^{-1}\sigma^{-3/2} \\
    \times a\left(\frac{R}{\sqrt{1-\sigma}},\theta,\frac{R'}{\sqrt{1-\sigma}},\theta'\right) b\left(\frac{R'}{\sqrt{\sigma}},\theta',\frac{R''}{\sqrt{\sigma}},\theta''\right) \, R'\, d\theta'\, dR'\, d\sigma.
\end{multline} 
Setting $(R'',\theta'')=(R,\theta)$ gives the little trace. If you also make the change of variables $\tilde R'=R'/\sqrt\sigma$, then $R' dR' = \sigma \tilde R' d \tilde R'$. Then
\begin{multline}\label{eq:horriblething}
    (\textrm{tr }H_{\text{sub}})(R,\theta) = -T^{-1} \int_0^1 \int_0^{\infty} \int_0^{\alpha} (1-\sigma)^{-1}\sigma^{-1/2} \\
    \times a\left(\frac{R}{\sqrt{1-\sigma}},\theta,\frac{\sqrt\sigma\tilde R'}{\sqrt{1-\sigma}},\theta'\right) b\left(\tilde R',\theta',\frac{R}{\sqrt\sigma},\theta\right) \, \tilde R'\, d\theta'\, d\tilde R'\, d\sigma.
\end{multline} 

To find the contribution to the heat trace, we first restrict this kernel to the diagonal by setting $(R',\theta')=(R,\theta)$, then we integrate with respect to $r\, dr\, d\theta = T^2 R \, dR\, d\theta$. 

The last thing to do is to address the dependence of this expression on $\kappa_{\pm}$. Let
\[
I_\alpha(\theta_0)
\]
denote the renormalized integral obtained from the right-hand side of \eqref{eq:general} by replacing
$b_{\mathrm{bis}}$ with the general function $b$ from \eqref{eq:bexpression-clean}. Then the corner
contribution produced by the conformal model is
\[
\mathcal C_{1/2}(\alpha,\kappa_+,\kappa_-)=r_0\,I_\alpha(\theta_0).
\]
By \eqref{eq:bexpression-clean}, every occurrence of $\theta_0$ is linear and enters only through shifted
angles such as $\theta+\theta_0$ and $\phi+\theta_0$. Expanding these shifted trigonometric functions shows
that
\[
I_\alpha(\theta_0)=A(\alpha)\cos\theta_0+B(\alpha)\sin\theta_0
\]
for some coefficients $A(\alpha)$ and $B(\alpha)$.

Now reflect the sector across its bisector. This swaps the two boundary rays, hence swaps
$\kappa_+$ and $\kappa_-$, but it does not change the heat trace coefficient because the reflected corner is
isometric to the original one. In terms of the parameter $\theta_0$, the reflection sends
\[
\theta_0\mapsto -\alpha-\theta_0,
\]
which is the transformation that exchanges the two formulas in \eqref{eq:kappas-via-theta0}. Therefore
\[
I_\alpha(\theta_0)=I_\alpha(-\alpha-\theta_0).
\]
Combining this symmetry with the linear form above gives
\[
I_\alpha(\theta_0)=C(\alpha)\cos\left(\theta_0+\frac{\alpha}{2}\right)
\]
for some function $C(\alpha)$. We now define
\[
c_{1/2}(\alpha):=I_\alpha\left(\pi-\frac{\alpha}{2}\right),
\]
which is exactly the renormalized trace formula \eqref{eq:general}. Since
\[
I_\alpha\left(\pi-\frac{\alpha}{2}\right)=-C(\alpha),
\]
we obtain
\[
I_\alpha(\theta_0)=-c_{1/2}(\alpha)\cos\left(\theta_0+\frac{\alpha}{2}\right).
\]
Finally, substituting \eqref{eq:bisector-amplitude} gives
\[
\mathcal C_{1/2}(\alpha,\kappa_+,\kappa_-)
=r_0 I_\alpha(\theta_0)
=c_{1/2}(\alpha)\,\frac{\kappa_+ + \kappa_-}{4\sin(\alpha/2)},
\]
which is \eqref{eq:form} in the case $\alpha\ne\pi$ (the case $\alpha=\pi$ is dealt with separately in section~\ref{ss:alpha_equals_pi}).

We give an explicit expression for $c_{1/2}(\alpha)$ obtained from \eqref{eq:horriblething}. By Melrose's pushforward theorem (see \cite{MelroseAPS}; see also \cite[Theorem 3.10]{GrieserBcalc} for a precise statement), the trace contribution \eqref{eq:horriblething} is then given by a \emph{renormalized} integral, which, for explicitness, we record as a proposition.

\begin{proposition}[Finite-part formula for $c_{1/2}(\alpha)$]
\label{prop:finite-part-c}
Assume $0<\alpha<2\pi$ and $\alpha\ne\pi$. Let $a$ be the sector kernel coefficient in \eqref{eq:a-def}, and let $b_{\mathrm{bis}}$ be the function in \eqref{eq:bexpression-clean} with
\[
\theta_0=\pi-\frac{\alpha}{2}.
\]
Then
\begin{equation}\label{eq:general}
\begin{aligned}
c_{1/2}(\alpha)
&= -\operatorname*{fp}_{A\to\infty}
\int_0^A \int_0^\alpha \int_0^1 \int_0^\infty \int_0^\alpha
(1-\sigma)^{-1}\sigma^{-1/2} \\
&\quad \times
a\left(\tfrac{R}{\sqrt{1-\sigma}},\theta,\tfrac{\sqrt{\sigma}\rho}{\sqrt{1-\sigma}},\varphi\right)
b_{\mathrm{bis}}\left(\rho,\varphi,\tfrac{R}{\sqrt{\sigma}},\theta\right)
\rho\,d\varphi\,d\rho\,d\sigma\,R\,d\theta\,dR.
\end{aligned}
\end{equation}
Here $\operatorname*{fp}_{A\to\infty}$ denotes the Hadamard finite part of the truncated integral.
\end{proposition}

The finite-part formula is separate from the sign argument in Section 5, where $c_{1/2}(\alpha)$ is represented instead by a domain variation and the heat Poisson kernel.

\section{The case $\alpha=\pi/2$}\label{sec:specialpi2}

\subsection{The Dirichlet case}
Our goal here is to compute through the $t^{1/2}$ term in the short-time expansion of the heat trace of the Laplacian on a domain $\Omega$ which is a half-disk of radius 1, with Dirichlet boundary conditions. This provides a concrete test case in which the corner term can be read off from an explicit spectrum and compared directly with \eqref{eq:heattrace} and \eqref{eq:form}. This heat trace is denoted
\begin{equation}\Tr e^{-t\Delta_{\Omega}}.\end{equation}
The spectrum of the Dirichlet Laplacian on $\Omega$ consists of $\{j_{m,k}^2\}_{m,k\in\mathbb N}$, each with multiplicity one. This may be compared to the spectrum of the Dirichlet Laplacian on a disk $D$, which consists of all of these eigenvalues but with multiplicity two, as well as the set $\{j_{0,k}^2\}_{k\in\mathbb N}$, each with multiplicity one. As a result, if we let
\begin{equation}
H(t) := \sum_{k=1}^{\infty}e^{-tj_{0,k}^2},
\end{equation}
then
\begin{equation}\label{eq:relationship}
\Tr e^{-t\Delta_{\Omega}} = \frac 12 \Tr e^{-t\Delta_D} - \frac 12H(t).
\end{equation}
We will take advantage of this.

The heat trace $\Tr e^{-t\Delta_D}$ is well understood and has a full polyhomogeneous asymptotic expansion in half-integer powers as $t\to 0$. The beginning of this expansion is, from \cite{EGS_2017},
\begin{equation}
\Tr e^{-t\Delta_D} = \frac{1}{4}t^{-1} -\frac{\sqrt\pi}{4}t^{-1/2}+\frac 16 +\frac{\sqrt\pi}{128}t^{1/2}+O(t).
\end{equation}

For $H(t)$, we use the asymptotics of Bessel function zeroes due to McMahon \cite[(10.21.19)]{NIST:DLMF}:
\begin{equation}
j_{0,k}= (k-\frac 14)\pi + \frac{1}{8(k-\frac 14)\pi} + O(k^{-3}),
\end{equation}
which imply
\begin{equation}
j_{0,k}^{2} = (k-\frac 14)^2\pi^2 + \frac 14 + O(k^{-2}).
\end{equation}
Based on this we define the function
\begin{equation}
\tilde H(t) = \sum_{k=1}^{\infty} \exp[-t((k-\frac 14)^2\pi^2 + \frac 14)]
\end{equation}
and estimate its difference with $H(t)$.
\begin{lemma}\label{lem:difference} With all notation as above, as $t\to 0$,
\begin{equation}
H(t) - \tilde H(t) = O(t).
\end{equation}
\end{lemma}
This is helpful because we also have the following lemma.
\begin{lemma}\label{lem:compute} As $t\to 0$,
\begin{equation}
\tilde H(t) = \frac{1}{2\sqrt\pi}t^{-1/2} + c - \frac{1}{8\sqrt\pi}t^{1/2}+ O(t).
\end{equation}
Here $c$ is a constant which is irrelevant to our purposes.
\end{lemma}
From these two Lemmas and \eqref{eq:relationship} we immediately deduce this proposition.
\begin{proposition} As $t\to 0$,
\begin{equation}
\Tr e^{-t\Delta_{\Omega}} = \frac 18 t^{-1} - (\frac{\sqrt\pi}{8}+\frac{1}{4\sqrt\pi})t^{-1/2}-\frac 12c + (\frac{\sqrt\pi}{256}+\frac{1}{16\sqrt\pi})t^{1/2}+ O(t).
\end{equation}
\end{proposition}
\begin{remark} We could reverse engineer $c$ from the known expansion from \cite{NRS} but there is no need.
\end{remark}

We are now in a position to state what the contribution of each corner of the half-disk to the $t^{1/2}$ term is. The $t^{1/2}$ term above is the sum of a curvature integral over the boundary and a pair of identical corner contributions. From \cite{EGS_2017} we see that the $\frac{\sqrt\pi}{256}$ is precisely the curvature integral. Dividing the remainder in half shows
\begin{theorem} The contribution to the $t^{1/2}$ coefficient in the heat trace of a half-disk of radius 1 from each of its two corners is
\begin{equation}
\frac{1}{32\sqrt\pi}t^{1/2}.
\end{equation}
\end{theorem}
Combined with \eqref{eq:form}, this identifies
\[
c_{1/2}\left(\frac{\pi}{2}\right)=\frac{\sqrt{2}}{16\sqrt{\pi}},
\]
and confirms that right-angled curved corners contribute nontrivially at order $t^{1/2}$. It remains only to prove the Lemmas.

\subsection{Proofs of the Lemmas}
First we prove Lemma \ref{lem:difference}.
\begin{proof} Rearrangement of the various expressions gives
\begin{equation}
H(t)-\tilde H(t) = e^{-t/4}\sum_{k=1}^{\infty}e^{-t(k-\frac 14)^2\pi^2}(1-e^{-t\cdot O(k^{-2})}).
\end{equation}
Here $O(k^{-2})$ is shorthand for $j_{0,k}^2-(k-\frac 14)^2\pi^2 - \frac 14$, which is bounded in absolute value by $Ck^{-2}$ for some universal constant $C$. Since $k\ge 1$ and $t\le 1$ is all that is relevant, we have that for $t\le 1$, $|1-e^{-t\cdot O(k^{-2})}|\le \tilde C tk^{-2}$ for a universal constant $\tilde C$. This implies
\begin{equation}
|H(t)-\tilde H(t)| \le \tilde Ct\sum_{k=1}^{\infty}e^{-t(k-\frac 14)^2\pi^2}k^{-2}.
\end{equation}
Estimating each exponential from above by 1 yields that the right-hand side is $O(t)$ as desired.
\end{proof}

Now we prove Lemma \ref{lem:compute}.
\begin{proof}
First write
\begin{equation}
\tilde H(t) = e^{-t/4}\hat H(t),
\end{equation}
where
\begin{equation}
\hat H(t) = \sum_{k=1}^{\infty}\exp[-t(k-\frac 14)^2\pi^2].
\end{equation}

We know from \cite{NRS} that $\Tr e^{-t\Delta_{\Omega}}$ has a full polyhomogeneous asymptotic expansion, and therefore by \eqref{eq:relationship}, so does $H(t)$. By Lemma \ref{lem:difference}, $\tilde H(t)$ has at least a polyhomogeneous expansion up to $O(t)$, and therefore so does $\hat H(t)$. The upshot of this is that one may define a zeta function
\begin{equation}
\hat \zeta(s) = \frac{1}{\Gamma(s)}\int_0^{\infty}\hat H(t) t^{s-1}\, dt,
\end{equation}
and that this zeta function, initially defined for $\Re(s)>>0$, has a meromorphic continuation at least to the half-plane $\Re(s) > -1$. Moreover, by the usual argument (agreement for $\Re(s)>>0$),
\begin{equation}
\hat\zeta(s) = \sum_{k=1}^{\infty}((k-\frac 14)\pi)^{-2s}.
\end{equation}

But now we recognize the zeta function $\hat\zeta(s)$ as a form of a Hurwitz zeta function. Specifically,
\begin{equation}
\hat\zeta(s) = \pi^{-2s}\zeta(2s,\frac 34),
\end{equation}
where $\zeta(s,a)$ is the Hurwitz zeta function. The function $\zeta(s,a)$ has a single, simple pole at $s=1$, with residue 1 \cite[25.11(i)]{NIST:DLMF}. So $\hat\zeta(s)$ has a single, simple pole at $s=1/2$, and its residue there is $\frac{1}{2\pi}$. Thus the function
\begin{equation}
\int_0^{\infty}\hat H(t)t^{s-1}\, dt = \Gamma(s)\hat\zeta(s)
\end{equation}
has poles at $s=1/2$ and also at each non-positive integer; moreover its residue at $s=1/2$ is $\Gamma(1/2)/(2\pi) = \frac{1}{2\sqrt{\pi}}$.

We can now reverse-engineer the expansion of $\hat H(t)$ as $t\to 0$, which is morally taking an inverse Mellin transform. We already know such a partial expansion exists up to order $O(t)$, and now we know that its coefficients are at orders $t^{-1/2}$, $t^0$, and higher integer powers of $t$ which are absorbed in $O(t)$. The coefficient of $t^{-1/2}$ is precisely the residue of $\Gamma(s)\hat\zeta(s)$ at $1/2$, which is $\frac{1}{2\sqrt{\pi}}$. So, for some constant $c$ which is not relevant for our purposes,
\begin{equation}
\hat H(t) = \frac{1}{2\sqrt{\pi}}t^{-1/2} + c + O(t).
\end{equation}
Multiplying by $e^{-t/4}\sim 1-\frac t4+O(t^2)$,
\begin{equation}
\tilde H(t)=\frac{1}{2\sqrt\pi}t^{-1/2} + c - \frac{1}{8\sqrt\pi}t^{1/2} + O(t),
\end{equation}
completing the proof.
\end{proof}

\subsection{A Neumann right-angle corner contribution}\label{app:neumann-pi2}

This section computes the order-$t^{1/2}$ contribution coming from the two right-angle corners of the
unit half-disk with Neumann boundary conditions. The computation uses the explicit Neumann spectrum on
the disk and half-disk, together with the smooth-boundary coefficient for Neumann boundary conditions.
The final outcome is an explicit value for the corner remainder in this model geometry, which we write
as $\mathcal C^N_{1/2}\bigl(\frac{\pi}{2},1,0\bigr)$.

Let $\Omega\subset\mathbb R^2$ be the unit half-disk
\[
\Omega=\{(r,\theta):0<r<1,\ 0<\theta<\pi\},
\]
endowed with Neumann boundary conditions along both the diameter and the semicircle. Let $D$ denote the
unit disk, also with Neumann boundary conditions, and write $\Delta^N_{\Omega}$ and $\Delta^N_{D}$ for the
corresponding Laplacians. For the unit disk with Neumann boundary conditions, the smooth-boundary heat
trace expansion gives (see Proposition 8 of \cite{EGS_2017})
\begin{equation}\label{eq:disk-Neumann-expansion}
\Tr e^{-t\Delta^N_{D}}
=\frac{1}{4}t^{-1}+\frac{\sqrt\pi}{4}t^{-1/2}+\frac16+\frac{5\sqrt\pi}{128}t^{1/2}+O(t).
\end{equation}

Separation of variables shows that the Neumann spectrum on the disk is
\[
\mathrm{spec}(\Delta_D^N)=\{0\}\ \cup\ \{(j'_{m,k})^2:\ m\in\mathbb N_0,\ k\in\mathbb N\},
\]
where $j'_{m,k}$ is the $k$th positive zero of $J_m'$, with multiplicity $1$ for $m=0$ and multiplicity $2$
for $m\ge1$ (the $\cos(m\theta)$ and $\sin(m\theta)$ modes). On the half-disk, the Neumann condition on
$\theta=0,\pi$ forces $\partial_\theta u=0$, hence only the $\cos(m\theta)$ modes occur, so each
$(j'_{m,k})^2$ has multiplicity $1$. Thus
\begin{equation}\label{eq:relationship-Neumann}
\Tr e^{-t\Delta^N_{\Omega}}=\frac12\,\Tr e^{-t\Delta^N_{D}}+\frac12\,H^N(t),
\end{equation}
where the correction term is precisely the $m=0$ contribution
\[
H^N(t):=\sum_{\lambda\in\{0\}\cup\{(j'_{0,k})^2\}_{k\ge1}}e^{-t\lambda}
=1+\sum_{k=1}^{\infty}e^{-t(j'_{0,k})^2}.
\]
Using $J_0'(x)=-J_1(x)$, this may be rewritten as
\begin{equation}\label{eq:HN-def}
H^N(t)=1+\sum_{k=1}^{\infty}e^{-t j_{1,k}^2},
\end{equation}
where $\{j_{1,k}\}_{k\ge1}$ are the positive zeros of $J_1$.

McMahon's asymptotics for Bessel zeros (cf.\ \cite[(10.21.19)]{NIST:DLMF}) give
\begin{equation}\label{eq:j1-mcmahon}
j_{1,k}=\Bigl(k+\frac14\Bigr)\pi-\frac{3}{8\bigl(k+\frac14\bigr)\pi}+O(k^{-3}),
\qquad k\to\infty,
\end{equation}
hence
\begin{equation}\label{eq:j1-square}
j_{1,k}^2=\Bigl(k+\frac14\Bigr)^2\pi^2-\frac34+O(k^{-2}).
\end{equation}
Define the model series
\begin{equation}\label{eq:HN-tilde}
\widetilde H^N(t):=
1+\sum_{k=1}^{\infty}\exp \Bigl[-t\Bigl(\Bigl(k+\frac14\Bigr)^2\pi^2-\frac34\Bigr)\Bigr].
\end{equation}

\begin{lemma}\label{lem:HN-difference}
As $t\to0$ one has $H^N(t)-\widetilde H^N(t)=O(t)$.
\end{lemma}

\begin{proof}
This is identical to the proof of Lemma \ref{lem:difference}.
\end{proof}

\begin{lemma}\label{lem:HN-compute}
As $t\to0$,
\begin{equation}\label{eq:HN-expansion}
\widetilde H^N(t)=\frac{1}{2\sqrt\pi}\,t^{-1/2}+c_N+\frac{3}{8\sqrt\pi}\,t^{1/2}+O(t),
\end{equation}
for some constant $c_N\in\mathbb R$.
\end{lemma}

\begin{proof}
Write $\widetilde H^N(t)=e^{3t/4}\,\widehat H^N(t)$ with
\[
\widehat H^N(t):=e^{-3t/4}\widetilde H^N(t)
=e^{-3t/4}+\sum_{k=1}^{\infty}\exp \Bigl[-t\Bigl(k+\frac14\Bigr)^2\pi^2\Bigr].
\]
As in the proof of Lemma \ref{lem:compute}, the existence of a polyhomogeneous expansion for
$\widehat H^N(t)$ through $O(t)$ allows one to define the Mellin transform
\[
\widehat\zeta_N(s)=\frac{1}{\Gamma(s)}\int_0^\infty \widehat H^N(t)\,t^{s-1}\,dt
=\sum_{k=0}^{\infty}\bigl(\pi(k+\tfrac54)\bigr)^{-2s}
=\pi^{-2s}\zeta(2s,\tfrac54),
\]
initially for $\Re(s)>>0$ and then by meromorphic continuation.

Since $\zeta(2s,\tfrac54)$ has a simple pole at $2s=1$ with residue $1$, the function $\widehat\zeta_N$
has a simple pole at $s=\tfrac12$ with residue $1/(2\pi)$. It follows that
$\int_0^\infty \widehat H^N(t)t^{s-1}\,dt=\Gamma(s)\widehat\zeta_N(s)$ has residue
$\Gamma(1/2)/(2\pi)=1/(2\sqrt\pi)$ at $s=1/2$, and hence
\[
\widehat H^N(t)=\frac{1}{2\sqrt\pi}\,t^{-1/2}+c_N+O(t).
\]
Multiplying by $e^{3t/4}=1+\frac34 t+O(t^2)$ yields \eqref{eq:HN-expansion}.
\end{proof}

Combining Lemmas \ref{lem:HN-difference} and \ref{lem:HN-compute} shows that $H^N(t)$ itself satisfies
\eqref{eq:HN-expansion} (with a possibly different constant term), up to $O(t)$.
Substituting \eqref{eq:disk-Neumann-expansion} and \eqref{eq:HN-expansion} into
\eqref{eq:relationship-Neumann} gives the half-disk expansion
\begin{equation}\label{eq:half-Neumann-expansion}
\Tr e^{-t\Delta^N_{\Omega}}
=\frac18\,t^{-1}
+\Bigl(\frac{\sqrt\pi}{8}+\frac{1}{4\sqrt\pi}\Bigr)t^{-1/2}
+\text{\emph{(const)}}
+\Bigl(\frac{5\sqrt\pi}{256}+\frac{3}{16\sqrt\pi}\Bigr)t^{1/2}
+O(t).
\end{equation}

The coefficient of $t^{1/2}$ in \eqref{eq:half-Neumann-expansion} may be split into a smooth-boundary term and a corner remainder by subtracting the smooth-boundary contribution for Neumann boundary conditions. For the unit disk, $\int_{\partial D}\kappa^2\,ds=2\pi$, and the $t^{1/2}$ coefficient $\frac{5\sqrt\pi}{128}$ in \eqref{eq:disk-Neumann-expansion} corresponds to the smooth-boundary constant $\frac{5}{256\sqrt\pi}$. For the half-disk, $\partial\Omega$ consists of a semicircle of curvature $1$ and length $\pi$ and a diameter of curvature $0$, so
\[
\frac{5}{256\sqrt\pi}\int_{\partial\Omega}\kappa^2\,ds=\frac{5}{256\sqrt\pi}\cdot\pi=\frac{5\sqrt\pi}{256}.
\]
Subtracting this from the $t^{1/2}$ coefficient in \eqref{eq:half-Neumann-expansion} leaves a remainder $\frac{3}{16\sqrt\pi}$. The two endpoints of the semicircle are congruent corners, so they contribute equally. This motivates the definition
\begin{equation}\label{eq:C12N-pi2}
\mathcal C^N_{1/2}\Bigl(\frac{\pi}{2},1,0\Bigr):=\frac12\cdot\frac{3}{16\sqrt\pi}
=\frac{3}{32\sqrt\pi}.
\end{equation}
Here $(\kappa_+,\kappa_-)=(1,0)$ records that one incident boundary arc has curvature $1$ (the semicircle) and the other has curvature $0$ (the diameter).

\begin{remark}
Equation \eqref{eq:C12N-pi2} records the corner value
$\mathcal C^N_{1/2}\bigl(\frac{\pi}{2},1,0\bigr)=\frac{3}{32\sqrt\pi}$.
If a Neumann factorization of the form
\[
\mathcal C^N_{1/2}(\alpha,\kappa_+,\kappa_-)=c^N_{1/2}(\alpha)\,\frac{\kappa_+ + \kappa_-}{4\sin(\alpha/2)}
\]
holds for $\alpha=\pi/2$, then since $(\kappa_+,\kappa_-)=(1,0)$ in \eqref{eq:C12N-pi2}, one obtains
\[
c^N_{1/2}\Bigl(\frac{\pi}{2}\Bigr)
=\frac{\mathcal C^N_{1/2}\bigl(\frac{\pi}{2},1,0\bigr)}{1/(2\sqrt{2})}
=\frac{3\sqrt{2}}{16\sqrt\pi}.
\]
\end{remark}

\section{A formula for $c_{1/2}(\alpha)$}\label{section:sign}

The goal of this section is to prove \eqref{eq:moreexplicitc} and \eqref{eq:sign}. To this end, let
\[
  W_\alpha=\{(r,\theta):r>0,\ 0<\theta<\alpha\},
  \qquad 0<\alpha<2\pi,
\]
and let
\[
  \HH=\{(r,\theta):r>0,\ 0<\theta<\pi\}
\]
be the half-plane tangent to $W_\alpha$ along the side $\theta=0$.

\begin{theorem}
\label{thm:mainsign} Equation \eqref{eq:sign} holds:
\[
  \operatorname{sgn} c_{1/2}(\alpha)=\operatorname{sgn}(\pi-\alpha),
  \qquad 0<\alpha<2\pi.
\]
\end{theorem}

The proof is carried out by constructing a specific family of domains $\Omega_{\eps}$ and taking the derivative of the heat trace at $\eps=0$.

In order to construct our family of domains, first consider the holomorphic map from $\mathbb C\to\mathbb C$ given by
\begin{equation}
\label{eq:Phi}
  \Phi_\eps(z)=z+\eps e^{i(\pi-\alpha/2)}z^2.
\end{equation}
We analyze the image of $W_{\alpha}$ under this map. First consider the lower side $\theta=0$.  A point on that side is $z=s$,
$s>0$.  Its image is
\[
  \gamma_+(s)=\Phi_\eps(s)=s+\eps e^{i(\pi-\alpha/2)}s^2.
\]
At $s=0$,
\[
  \gamma_+'(0)=1,
  \qquad
  \gamma_+''(0)=2\eps e^{i(\pi-\alpha/2)}.
\]
The inward unit normal to the lower side is $i=(0,1)$.  Since
$|\gamma_+'(0)|=1$, the inward curvature at the vertex is
\begin{equation}
\label{eq:kappa-plus}
  \kappa_+=\gamma_+''(0)\cdot i
  =2\eps\sin(\pi-\alpha/2)
  =2\eps\sin(\alpha/2).
\end{equation}

On the upper side $\theta=\alpha$, a point is $z=se^{i\alpha}$.  Its
image is
\[
  \gamma_-(s)=\Phi_\eps(se^{i\alpha})
  =se^{i\alpha}+\eps s^2e^{i(\pi+3\alpha/2)}.
\]
At $s=0$,
\[
  \gamma_-'(0)=e^{i\alpha},
  \qquad
  \gamma_-''(0)=2\eps e^{i(\pi+3\alpha/2)}.
\]
The inward unit normal to the upper side is the clockwise rotation of
$e^{i\alpha}$, namely $e^{i(\alpha-\pi/2)}$.  Hence
\begin{align}
\label{eq:kappa-minus}
  \kappa_-
  &=2\eps\cos\bigl((\pi+3\alpha/2)-(\alpha-\pi/2)\bigr)  \\
  &=2\eps\cos(3\pi/2+\alpha/2)
    =2\eps\sin(\alpha/2).
\end{align}
Thus
\begin{equation}
\label{eq:kappa-symmetric}
  \kappa_+=\kappa_-=2\eps\sin(\alpha/2).
\end{equation}
Substitution into \eqref{eq:form} gives
\begin{equation}
\label{eq:C-equals-eps-c}
  \calC_{1/2}(\alpha,\kappa_+,\kappa_-)
  =c_{1/2}(\alpha)\frac{4\eps\sin(\alpha/2)}{4\sin(\alpha/2)}
  =\eps c_{1/2}(\alpha).
\end{equation}
Therefore $c_{1/2}(\alpha)$ is the derivative, at $\eps=0$, of the local curved-corner coefficient for this symmetric deformation.

We construct $\Omega_{\eps}$ by using this deformation as a local model. To do this, it is helpful to think in terms of the normal velocity of the deformation. On the
lower side,
\[
  \left.\partial_\eps\Phi_\eps(s)\right|_{\eps=0}
  =s^2e^{i(\pi-\alpha/2)}
  =s^2\bigl(-\cos(\alpha/2)+i\sin(\alpha/2)\bigr).
\]
The outward unit normal to $W_\alpha$ along $\theta=0$ is $-i=(0,-1)$.  Hence
\begin{equation}
\label{eq:V-lower}
  V_0(s)=s^2e^{i(\pi-\alpha/2)}\cdot(-i)
  =-s^2\sin(\alpha/2)
\end{equation}
where $V_0$ is the outward normal velocity of the boundary deformation along the lower side of the wedge $W_\alpha$.

On the upper side,
\[
  \left.\partial_\eps\Phi_\eps(se^{i\alpha})\right|_{\eps=0}
  =s^2e^{i(\pi+3\alpha/2)}.
\]
The outward unit normal to $W_\alpha$ along $\theta=\alpha$ is the counterclockwise rotation of $e^{i\alpha}$, namely $e^{i(\alpha+\pi/2)}$.
Therefore
\begin{align}
\label{eq:V-upper}
  V_\alpha(s)
  &=s^2\cos\bigl((\pi+3\alpha/2)-(\alpha+\pi/2)\bigr) \\
  &=s^2\cos(\pi/2+\alpha/2)
    =-s^2\sin(\alpha/2).
\end{align}
Thus both sides have the same outward velocity
\begin{equation}
\label{eq:V-both}
  V(s)=-s^2\sin(\alpha/2).
\end{equation}
For $0<\alpha<2\pi$, the factor $\sin(\alpha/2)$ is positive.  Hence, for $\eps>0$, the deformation has positive inward curvature and moves both sides inward.

Finally, fix $\Omega=\Omega_0$ to be a curvilinear polygon with one corner at the origin, which is isometric to $W_{\alpha}$ on a ball of radius 1 about the origin. Let $\chi(R):\mathbb R_+\to\mathbb R_+$ be a smooth cutoff function, supported on $\{R<1\}$ and identically 1 on $\{R<1/2\}$. Then let $\Omega_{\eps}$ be the family of domains obtained by taking $\Omega_0$ and applying the normal velocity 
\[\chi(s)V(s).\]
Outside the ball $\{R\le 1\}$, $\Omega_{\eps}$ is isometric to $\Omega_0$. On the ball $\{R\le\frac 12\}$, $\Omega_{\eps}$ is isometric to the deformation of the wedge by the holomorphic map $\Phi_{\eps}$. In particular, we still have \eqref{eq:C-equals-eps-c}: $c_{1/2}(\alpha)$ is the derivative at $\eps=0$ of the local curved-corner coefficient for the heat trace of $\Omega_{\eps}$.

To get a handle on this heat trace, we consider a more general domain deformation setup. For any Lipschitz domain $\Omega\subseteq\mathbb R^2$, and any point $q\in\partial\Omega$ where the outward normal $\nu_q$ is defined, we let the positive Dirichlet heat Poisson kernel be
\begin{equation}
\label{eq:Poisson-def}
  \PH_\Omega(t;q,z)=-\partial_{\nu_q}H_\Omega(t;q,z),
  \qquad z\in \Omega.
\end{equation}
The sign is chosen so that $\PH_D$ is positive. Now let
\begin{equation}
\label{eq:B-def}
  \BH_\Omega(t,q)=\int_\Omega \PH_\Omega(t/2;q,z)^2\,dz.
\end{equation}

\begin{proposition}
    For a smoothly varying family $\Omega_{\eps}$ of bounded curvilinear polygons, with normal velocity $V(q)$, we have
    \begin{equation}
\label{eq:epsderiv-of-trace}
  \left.\frac{d}{d\eps}\Tr e^{-t\Delta^D_{\Omega_\eps}}\right|_{\eps=0}
  =t\int_{\partial\Omega_0}V(q)\BH_{\Omega_0}(t,q)\,dq.
\end{equation}
\end{proposition}
\begin{remark}
    All integrals are taken over the regular part of the boundary, away from the vertices. The vertices, being a finite set, contribute nothing to the right-hand side.
\end{remark}
\begin{remark}
    An outward motion increases the heat trace, as it should.
\end{remark}
\begin{proof}
Let $\{u_j\}$ be an orthonormal Dirichlet eigenbasis of $\Omega$, with $-\Delta u_j=\lambda_j u_j$.  Then
\[
  H_\Omega(t;x,y)=\sum_j e^{-t\lambda_j}u_j(x)u_j(y),
\]
and hence for any $q$ on $\partial\Omega$ which is not a vertex,
\[
  \PH_\Omega(t/2;q,z)
  =\sum_j e^{-t\lambda_j/2}\bigl(-\partial_\nu u_j(q)\bigr)u_j(z).
\]
Squaring and integrating over $z$ gives, by orthonormality,
\begin{equation}
\label{eq:B-spectral}
  \BH_\Omega(t,q)=\sum_j e^{-t\lambda_j}\bigl(\partial_\nu u_j(q)\bigr)^2.
\end{equation}
Note also that
\[
  \partial_{\nu_x}\partial_{\nu_y}H_\Omega(t;x,y)|_{x=y=q}
  =\sum_j e^{-t\lambda_j}\bigl(\partial_\nu u_j(q)\bigr)^2.
\]
So
\begin{equation}
\label{eq:B-normal}
  \BH_\Omega(t,q)=\partial_{\nu_x}\partial_{\nu_y}H_\Omega(t;x,y)|_{x=y=q}.
\end{equation}

The Dirichlet eigenvalue variation formula, which holds in the general Lipschitz setting, now says
\begin{equation}
\label{eq:eigenvalue-hadamard}
  \lambda_j'(0)=-\int_{\partial\Omega_0}V(q)
  \bigl(\partial_\nu u_j(q)\bigr)^2\,dq.
\end{equation}
The sign is correct because an outward deformation lowers Dirichlet eigenvalues. Differentiating the heat trace gives
\begin{align}
\label{eq:heat-hadamard-derivation}
  \left.\frac{d}{d\eps}\Tr e^{-t\Delta^D_{\Omega_\eps}}\right|_{\eps=0}
  &=-t\sum_j e^{-t\lambda_j}\lambda_j'(0)  \\
  &=t\int_{\partial\Omega_0}V(q)
    \sum_j e^{-t\lambda_j}\bigl(\partial_\nu u_j(q)\bigr)^2\,dq.
\end{align}
Using \eqref{eq:B-spectral}, this yields the desired statement. \end{proof}

\begin{proof}[Proof of Theorem \ref{thm:mainsign}]
We now specialize to our family $\Omega_{\eps}$ constructed earlier and analyze both sides of \eqref{eq:epsderiv-of-trace}.

First, we claim that the left-hand side, and hence both sides, have a complete asymptotic expansion as $t\to 0$. This follows from \cite{NRS}, as the construction there may be repeated verbatim with the smoothly varying parameter $\eps$. (The domain deformation may be interpreted as a variation of the underlying metric). The heat trace $\Tr e^{-t\Delta_{\Omega_{\eps}}^D}$ thus has an expansion as $t\to 0$ with all coefficients smooth in $\eps$. Taking the $\varepsilon$-derivative at $\eps=0$ of this expansion yields a complete asymptotic expansion of the left-hand side.

Now we discuss the character of the expansion. The coefficients of the expansion of the left-hand side of \eqref{eq:epsderiv-of-trace} are precisely the first variations of the coefficients of the expansion of $\Tr e^{-t\Delta_{\Omega_{\eps}}^D}$. In other words, the heat-trace expansion of the family has the form
\[
\Tr e^{-t\Delta^D_{\Omega_\eps}} \sim \sum_k c_k(\eps)\, t^{a_k}(\log t)^{b_k},
\]
in which the indices $(a_k, b_k)$ are independent of $\eps$ and the coefficients $c_k(\eps)$ depend smoothly on $\eps$. Differentiation at $\eps = 0$ therefore acts on the coefficients alone:
\[
\frac{d}{d\eps}\Tr e^{-t\Delta^D_{\Omega_\eps}}\bigg|_{\eps=0} \sim \sum_k c_k'(0)\, t^{a_k}(\log t)^{b_k}.
\]
The coefficient at each order on the left-hand side of \eqref{eq:epsderiv-of-trace} is thus the first variation, at $\eps=0$, of the corresponding heat invariant of $\Omega_\eps$. We now compute these variations order by order.
\begin{itemize}
    \item The leading term, $t^{-1}$, is the area variation, which is
\[
\frac{1}{4\pi t}\int_{\partial\Omega}\chi(q)V(q)\, dq.
\]
\item The coefficient of $t^{-1/2}$ is the first variation of the boundary length. This is zero, since the varying part of $\partial\Omega_0$ is initially straight and the cutoff kills the endpoint terms.
\item The coefficient of $t^0$ is the variation of the integral of the boundary curvature, plus the variation of the angle. But the angle is fixed, and Gauss-Bonnet fixes the integral of the boundary curvature. Thus this first variation is zero.
\item The coefficient of $t^{1/2}$ is the first variation of the corner contribution, plus the first variation of $\int\kappa^2$. The latter is zero, since $\kappa=0$ on the part of $\Omega_0$ that varies. Hence the variation of $\kappa^2$ is $2\dot\kappa\kappa=0$.  Thus the coefficient of $t^{1/2}$ is the first variation of the corner contribution, which by \eqref{eq:C-equals-eps-c} is $c_{1/2}(\alpha)$.
\end{itemize}
All in all, we have just shown that the left-hand side of \eqref{eq:epsderiv-of-trace} satisfies
\begin{equation}\label{eq:lhsexpansion}
      \left.\frac{d}{d\eps}\Tr e^{-t\Delta^D_{\Omega_\eps}}\right|_{\eps=0} =
      \frac{1}{4\pi t}\int_{\partial\Omega}\chi(q)V(q)\, dq
      + c_{1/2}(\alpha)t^{1/2} + O(t\log t).
\end{equation}

Now we analyze the right-hand side of \eqref{eq:epsderiv-of-trace}. First make the observation that
\[
t\int_{\partial\Omega} \chi(q)V(q)\BH_{\Omega_0}(t,q)
-
t\int_{\partial\Omega}\chi(q)V(q)\BH_{W_{\alpha}}(t,q)
=
O(t^{\infty}).
\]
This is true because $W_{\alpha}$ and $\Omega_0$ are isometric on the support of $\chi$ and because all relevant heat kernels decay rapidly away from the diagonal. (Indeed the error is exponentially small, but we do not need this.) We now add and subtract a half-plane term along each of the two sides:
\[
t\int_{\partial\Omega} \chi(q)V(q)\BH_{\Omega_0}(t,q)
=
t\int_{\partial\Omega} \chi(q)V(q)(\BH_{W_{\alpha}}-\BH_{\HH})(t,q)
+
t\int_{\partial\Omega} \chi(q)V(q)\BH_{\HH}(t,q)
+
O(t^{\infty}).
\]
By a direct calculation with the half-plane heat kernel,
\begin{equation}\label{eq:rhsexpansion}
t\int_{\partial\Omega} \chi(q)V(q)\BH_{\Omega_0}(t,q)
=
t\int_{\partial\Omega} \chi(q)V(q)(\BH_{W_{\alpha}}-\BH_{\HH})(t,q)
+
\frac{1}{4\pi t}\int_{\partial\Omega} \chi(q)V(q)\, dq
+
O(t^{\infty}).
\end{equation}
Comparing \eqref{eq:lhsexpansion} and \eqref{eq:rhsexpansion}, using \eqref{eq:epsderiv-of-trace}, yields
\begin{equation}\label{eq:comparedexpansions}
t\int_{\partial\Omega} \chi(q)V(q)(\BH_{W_{\alpha}}-\BH_{\HH})(t,q) =t^{1/2}c_{1/2}(\alpha) + O(t\log t).\end{equation}
After subtracting the half-plane contribution, the $t^{-1}$ term cancels. The $t^{-1/2}$ and $t^0$ variations vanish for this deformation, so the first remaining term is the variation of the curved-corner coefficient, namely $c_{1/2}(\alpha)t^{1/2}$.

But now observe the following consequence of domain monotonicity:
\begin{lemma}
\label{lem:B-monotone}
Let $D_1\subsetneq D_2$ be two connected domains, and suppose that near a
boundary point $q$ they share the same smooth boundary side and the same outward
normal.  Then, for every $t>0$,
\[
  \BH_{D_1}(t,q)<\BH_{D_2}(t,q),
\]
where $\BH_D(t,q)=\|\PH_D(t/2;q,\cdot)\|_{L^2(D)}^2$.
\end{lemma}

\begin{proof}
Dirichlet heat kernels are monotone under domain inclusion:
\begin{equation}
\label{eq:kernel-monotone}
  0<H_{D_1}(\tau;x,z)\le H_{D_2}(\tau;x,z),
  \qquad x,z\in D_1,
\end{equation}
for every $\tau>0$.  This follows from the killed Brownian motion representation, or from the maximum principle.

Fix $z\in D_1$ and set
\[
  u(\tau,x)=H_{D_2}(\tau;x,z)-H_{D_1}(\tau;x,z),
  \qquad x\in D_1.
\]
Then $u\ge0$ and $u$ satisfies the heat equation in $(0,\infty)\times D_1$.  At the common boundary point $q$, both heat kernels vanish, so $u(\tau,q)=0$.  The parabolic Hopf lemma gives
\[
  \partial_{n_{\rm in}}u(\tau,q)>0,
\]
unless $u$ is identically zero.  It is not identically zero because $D_2$ is a proper larger connected domain.  Since
\[
  \PH_D(\tau;q,z)=\partial_{n_{\rm in}}H_D(\tau;q,z),
\]
we obtain
\begin{equation}
\label{eq:P-monotone}
  0<\PH_{D_1}(\tau;q,z)<\PH_{D_2}(\tau;q,z),
  \qquad z\in D_1.
\end{equation}
Extend $\PH_{D_1}$ by zero to $D_2\setminus D_1$.  With $\tau=t/2$,
\begin{align*}
  \BH_{D_2}(t,q)-\BH_{D_1}(t,q)
  &=\int_{D_2}\PH_{D_2}(\tau;q,z)^2\,dz
    -\int_{D_1}\PH_{D_1}(\tau;q,z)^2\,dz  \\
  &=\int_{D_1}\bigl(\PH_{D_2}^2-\PH_{D_1}^2\bigr)(\tau;q,z)\,dz
    +\int_{D_2\setminus D_1}\PH_{D_2}(\tau;q,z)^2\,dz.
\end{align*}
The first integrand is positive on a set of positive measure, and the second term is nonnegative.  Hence the difference is strictly positive.
\end{proof}

This allows us to complete the proof of Theorem~\ref{thm:mainsign} whenever $\alpha\ne\pi$. Assume first that $0<\alpha<\pi$.  Then
\[
  W_\alpha\subset \HH.
\]
The two domains share the side $\theta=0$ near $q_s=(s,0)$, and Lemma
\ref{lem:B-monotone} gives
\begin{equation}
\label{eq:B-convex-order}
  \BH_{W_\alpha}(1,s)<\BH_\HH(1,s),
  \qquad s>0.
\end{equation}
Therefore the integrand in \eqref{eq:comparedexpansions} is strictly positive for every $s>0$, and thus $c_{1/2}(\alpha)>0$. The reverse argument immediately gives the case $\pi<\alpha<2\pi$. 
\end{proof}

Finally, we prove \eqref{eq:moreexplicitc}. Let $q_s=(s,0)$, and set
\[
D_\alpha(s)=\BH_{\HH}(1,q_s)-\BH_{W_\alpha}(1,q_s).
\]
The ray $\theta=\alpha$ gives the same contribution, after reflecting the sector across its bisector. Also, for any cone $C$,
\[
\BH_C(t,q_s)=t^{-2}\BH_C(1,q_{s/\sqrt t}).
\]
Using $V(s)=-s^2\sin(\alpha/2)$, we get
\[
\begin{aligned}
t\int_{\partial\Omega}\chi V(\BH_{W_\alpha}-\BH_{\HH})
&=
2t\int_0^\infty \chi(s)(-s^2\sin(\alpha/2))
(\BH_{W_\alpha}-\BH_{\HH})(t,q_s)\,ds  \\
&=
2\sin(\alpha/2)t^{1/2}
\int_0^\infty \chi(\sqrt t\,\sigma)\sigma^2D_\alpha(\sigma)\,d\sigma.
\end{aligned}
\]
Thus \eqref{eq:comparedexpansions} gives
\[
2\sin(\alpha/2)
\int_0^\infty \chi(\sqrt t\,\sigma)\sigma^2D_\alpha(\sigma)\,d\sigma
=
c_{1/2}(\alpha)+O(t^{1/2}\log t).
\]
Since $D_\alpha(\sigma)$ decays rapidly as $\sigma\to\infty$, the cutoff drops out in the limit. Hence we obtain \eqref{eq:moreexplicitc}:
\begin{equation}
c_{1/2}(\alpha)
=
2\sin(\alpha/2)\int_0^\infty s^2
\bigl(\BH_{\HH}(1,q_s)-\BH_{W_\alpha}(1,q_s)\bigr)\,ds, \qquad \alpha\in(0,2\pi).
\end{equation}

In the case $\alpha=\pi/N$, the formula for $\int_0^\infty s^2
\bigl(\BH_{\HH}(1,q_s)-\BH_{W_\alpha}(1,q_s)\bigr)\,ds$ can be evaluated by images. Let
\[
\theta_k=\frac{\pi k}{N},
\qquad k=1,\ldots,N-1.
\]
A straightforward computation with the Dirichlet image formula for $W_{\pi/N}$, one image at a time, gives, along the ray $\theta=0$,
\begin{equation}\label{eq:B-image-difference}
\BH_{\HH}(1,q_s)-\BH_{W_{\pi/N}}(1,q_s)
=
\frac{1}{4\pi}
\sum_{k=1}^{N-1}
e^{-s^2\sin^2\theta_k}
\left(
2s^2\sin^2\theta_k\cos^2\theta_k
-\cos(2\theta_k)
\right).
\end{equation}
Specifically, for $\alpha=\pi/N$ we use the finite image formula for the Dirichlet heat kernel in a sector \cite{Carslaw1898}. Let $\theta_k=\pi k/N$. Differentiating the image formula as in \eqref{eq:B-normal}, and evaluating at $x=y=q_s$ on the ray $\theta=0$, gives 
\[ \partial_{\nu_x}\partial_{\nu_y}H_{W_{\pi/N}}(1;x,y)|_{x=y=q_s} = \frac{1}{4\pi} + \frac{1}{4\pi}\sum_{k=1}^{N-1} e^{-s^2\sin^2\theta_k} \left( \cos(2\theta_k) - 2s^2\sin^2\theta_k\cos^2\theta_k \right). 
\] 
Since $\BH_{\HH}(1,q_s)=1/(4\pi)$ by \eqref{eq:B-normal}, subtraction gives \eqref{eq:B-image-difference}. Substitution into $\int_0^\infty s^2
\bigl(\BH_{\HH}(1,q_s)-\BH_{W_\alpha}(1,q_s)\bigr)\,ds$ and the elementary identities
\[
\int_0^\infty s^2 e^{-as^2}ds
=
\frac{\sqrt{\pi}}{4a^{3/2}},
\qquad
\int_0^\infty s^4 e^{-as^2}ds
=
\frac{3\sqrt{\pi}}{8a^{5/2}}
\]
then give
\[
\begin{aligned}
I_{\pi/N}
&=
\frac{1}{4\pi}
\sum_{k=1}^{N-1}
\left[
2\sin^2\theta_k\cos^2\theta_k
\frac{3\sqrt{\pi}}{8\sin^5\theta_k}
-
\cos(2\theta_k)
\frac{\sqrt{\pi}}{4\sin^3\theta_k}
\right]  \\
&=
\frac{1}{16\sqrt{\pi}}
\sum_{k=1}^{N-1}
\frac{1+\cos^2\theta_k}{\sin^3\theta_k},
\end{aligned}
\]
which is \eqref{eq:I-pi-over-N}.

\subsection{The case $\alpha=\pi$}\label{ss:alpha_equals_pi}
It remains to identify the local coefficient at a marked straight angle. Write
\[
  C_\pi(\kappa_+,\kappa_-)
  :=
  \mathcal C_{1/2}(\pi,\kappa_+,\kappa_-).
\]
By locality and dilation invariance of the heat expansion, $C_\pi$ is homogeneous of degree one:
\[
  C_\pi(\lambda\kappa_+,\lambda\kappa_-)
  =
  \lambda C_\pi(\kappa_+,\kappa_-),
  \qquad \lambda>0.
\]

Let $\Omega_0$ agree near the marked point with
\[
  \HH=\{(s,u):u>0\}.
\]
Choose $V$ compactly supported on the boundary line, smooth away from $s=0$, with $V(0)=V'(0)=0$, and with
\[
  V(s)=-\frac a2s^2 \quad (s>0),
  \qquad
  V(s)=-\frac b2s^2 \quad (s<0)
\]
near $0$. Set
\[
  \Omega_\eps=\{(s,u):u>-\eps V(s)\}
\]
near the marked point and leave the rest of the boundary fixed. Then the marked point has angle $\pi$ and one-sided inward curvatures
\[
  \kappa_+(\eps)=\eps a,
  \qquad
  \kappa_-(\eps)=\eps b .
\]
For the half-plane,
\[
  \BH_{\HH}(t,s)=\frac{1}{4\pi t^2}.
\]
Using \eqref{eq:epsderiv-of-trace}, and replacing the local domain by the half-plane up to an exponentially small error, gives
\[
  \left.\frac{d}{d\eps}\Tr e^{-t\Delta^D_{\Omega_\eps}}\right|_{\eps=0}
  =
  \frac{1}{4\pi t}\int V(s)\,ds+O(e^{-c/t}).
\]
There is no $t^{1/2}$ term. The first variation at $\eps=0$ of
\[
  \frac{1}{256\sqrt\pi}
  \int_{\partial\Omega_\eps}\kappa_\eps^2\,ds
\]
also vanishes. Hence
\[
  0
  =
  \left.\frac{d}{d\eps}\right|_{\eps=0}
  C_\pi(\eps a,\eps b)
  =
  C_\pi(a,b),
\]
where the last equality uses homogeneity. Since $a$ and $b$ are arbitrary,
\[
  \mathcal C_{1/2}(\pi,\kappa_+,\kappa_-)=0 .
\]

\bibliographystyle{plain}  %
\bibliography{bibliography}

@article{EGS_2017,
  title = {Spectral determination of semi-regular polygons},
  author = {A.~Enciso and J.~G{\'o}mez-Serrano},
  journal = {Journal of Differential Geometry},
  volume = {122},
  number = {3},
  pages = {399--419},
  year = {2022},
  month = nov,
  doi = {10.4310/jdg/1675712993},
  url = {https://doi.org/10.4310/jdg/1675712993}
}

@article{kac1966can,
  title={Can one hear the shape of a drum?},
  author={M.~Kac},
  journal={The american mathematical monthly},
  volume={73},
  number={4P2},
  pages={1--23},
  year={1966},
  publisher={Taylor \& Francis}
}

@article{mckean1967curvature,
  title={Curvature and the eigenvalues of the Laplacian},
  author={H.~McKean Jr and I.~Singer},
  journal={Journal of Differential Geometry},
  volume={1},
  number={1-2},
  pages={43--69},
  year={1967},
  publisher={Lehigh University}
}

@book{MelroseAPS,
  author    = {R.~Melrose},
  title     = {The Atiyah-Patodi-Singer Index Theorem},
  year      = {1993},
  edition   = {1st},
  publisher = {A K Peters/CRC Press},
  doi       = {10.1201/9781439864609},
  url       = {https://doi.org/10.1201/9781439864609}
}

@article{Cheeger1983,
  author   = {J.~Cheeger},
  title    = {Spectral geometry of singular {R}iemannian spaces},
  journal  = {J. Differential Geom.},
  fjournal = {Journal of Differential Geometry},
  volume   = {18},
  number   = {4},
  year     = {1983},
  pages    = {575--657}
}

@incollection{GrieserBcalc,
  title={Basics of the b-calculus},
  author={D.~Grieser},
  booktitle={Approaches to Singular Analysis: A Volume of Advances in Partial Differential Equations},
  pages={30--84},
  year={2001},
  publisher={Springer}
}

@misc{Grieser,
  author       = {D.~Grieser},
  title        = {Notes on heat kernel asymptotics},
  howpublished = {Lecture notes},
  year         = {2004},
  note         = {Available at \url{https://web.math.ku.dk/grubb/notes/heat.pdf}.},
  url          = {https://web.math.ku.dk/grubb/notes/heat.pdf}
}

@article{NRS,
  author  = {M.~Nursultanov and J.~Rowlett and D.~A.~Sher},
  title   = {The heat kernel on curvilinear polygonal domains in surfaces},
  journal = {Annales mathématiques du Québec},
  volume  = {49},
  number  = {1},
  year    = {2025},
  pages   = {1--61},
  doi     = {10.1007/s40316-024-00237-4}
}

@article{vanDenBergSrisat88,
  author  = {M.~van~den~Berg and S.~Srisatkunarajah},
  title   = {Heat equation for a region in $\mathbb{R}^2$ with a polygonal boundary},
  journal = {J. London Math. Soc. (2)},
  volume  = {37},
  number  = {1},
  year    = {1988},
  pages   = {119--127},
}

@article{Carslaw1898,
  author = {Carslaw, H. S.},
  title = {Some multiform solutions of the partial differential equations of physical mathematics and their applications},
  journal = {Proceedings of the London Mathematical Society},
  series = {1},
  volume = {30},
  number = {1},
  pages = {121--165},
  year = {1898},
  doi = {10.1112/plms/s1-30.1.121}
}

@misc{NIST:DLMF,
         key = "{\relax DLMF}",
       title = "{\it NIST Digital Library of Mathematical Functions}",
howpublished = "\url{https://dlmf.nist.gov/}, Release 1.2.4 of 2025-03-15",
         url = "https://dlmf.nist.gov/",
        note = "F.~W.~J. Olver, A.~B. {Olde Daalhuis}, D.~W. Lozier, B.~I. Schneider,
                R.~F. Boisvert, C.~W. Clark, B.~R. Miller, B.~V. Saunders,
                H.~S. Cohl, and M.~A. McClain, eds."}

@article{branson1990asymptotics,
  title={The asymptotics of the {L}aplacian on a manifold with boundary},
  author={T.~Branson and P.~Gilkey},
  journal={Communications in partial differential equations},
  volume={15},
  number={2},
  pages={245--272},
  year={1990},
  publisher={Taylor \& Francis}
}

\end{document}